\documentclass{article}
\usepackage[utf8]{inputenc}
\usepackage{amsmath,amsthm,amssymb}

\newcommand{\bbR}{\mathbb{R}}

\newcommand{\bbC}{\mathbb{C}}
\newcommand{\bbZ}{\mathbb{Z}}

\newcommand{\idd}{\sqrt{-1}\partial\bar\partial}

\DeclareMathOperator{\Vol}{Vol}

\DeclareMathOperator{\supp}{supp}

\DeclareMathOperator*{\esssup}{ess\,sup}

\newtheorem{lemma}{Lemma}
\newtheorem{theorem}{Theorem}
\newtheorem{proposition}{Proposition}
\newtheorem{question}{Question}

\newtheorem{conjecture}{Conjecture}

\newtheorem{definition}{Definition}\theoremstyle{definition}
\newtheorem{remark}{Remark}\theoremstyle{definition}

\title{Mutual asymptotic Fekete sequences}
\author{Jakob Hultgren}
\date{}

\begin{document}

\maketitle

\abstract{
A sequence of point configurations on a compact complex manifold is asymptotically Fekete if it is close to maximizing a sequence of Vandermonde determinants. These Vandermonde determinants are defined by tensor powers of a Hermitian ample line bundle and the point configurations in the sequence possess good sampling properties with respect to sections of the line bundle. In this paper, given a collection of toric Hermitian ample line bundles, we give necessary and sufficient condition for existence of a sequence of point configurations which is asymptotically Fekete (and hence possess good sampling properties) with respect to each one of the line bundles. When they exist, we also present a way of constructing such sequences. As a byproduct we get a new equidistribution property for maximizers of products of Vandermonde determinants.
%A key step in the proof is describing the asymptotics of an $\mathcal L$-type functional and a byproduct of this is a new equidistribution property for maximizers of products of Vandermonde determinants. 
}

\section{Introduction}\label{sec:Introduction}
\subsection{Fekete Configurations}\label{sec:IntroFek}
Let $X$ be a compact Kähler manifold with a Kähler form $\omega$. Let $L$ be an ample line bundle over $X$ and $\phi$ be the weight of a continuous Hermitian metric $|\cdot|$ on $L$. For each positive integer $k$, let $H^0(X,kL)$ be the space of holomorphic sections of the $k$'th tensor power $kL$ of $L$. Letting $N=N_{kL}$ denote the dimension of $H^0(X,kL)$, fixing a basis $\{s_1,\ldots,s_{N}\}$ for $H^0\left(X,kL\right)$ determines a function $|D|^2_\phi = |D^{(k)}|^2_\phi$ on $X^N$ given by the Vandermonde determinant
\begin{equation}
    \label{eq:VandermondeExplicit}
    |D(x_1,\ldots,x_{N})|^2_\phi = |\det(s_i(x_l))_{i,l}|^2e^{-2k \sum_{l=1}^{N} \phi(x_l)}.
\end{equation}
Here we have used the notational convention of identifying each section $s_i$ as well as the weight $\phi$ with its local representatives. These representatives are only valid locally, but the expression in \eqref{eq:VandermondeExplicit} is well-defined on $X^N$ (See for example Section~2.1 in \cite{BB} for details). Put differently, the basis $\{s_1,\ldots,s_{N}\}$ determines a section on the determinant bundle of the direct product vector bundle $(kL)^{\times N}$ over $X^{N}$ and \eqref{eq:VandermondeExplicit} is the pointwise norm of this section with respect to the Hermitian metric induced by $\phi$. 

Note that, by compactness of $X^{N}$, 
$$ \sup_{X^{N}}|D|^2_{\phi} = \max_{X^N} |D|^2_{\phi} <\infty. $$ 
Fekete points, or \emph{Fekete configurations}, of $(L,\phi)$ are elements $(x_1,\ldots,x_N)\in X^N$ which maximize $|D|^2_\phi$. It is easy to verify that although $|D|^2_\phi$ depends on the choice of basis for $H^0(X,kL)$, the Fekete configurations don't. When $X=\mathbb P^1$ and $L=\mathcal O(1)$, Fekete configurations coincide with the equilibrium states of a systems of particles with Coulomb interactions. In the general case, Fekete configurations can be thought of as optimal sampling configurations. Indeed, they maximize the norm of the determinant of the evaluation map $e_{(x_1,\ldots,x_N)}$ on $H^0(X,kL)$
$$ e_{(x_1,\ldots,x_N)}(s) = (s(x_1),\ldots,s(x_L)). $$
Consequently, they provide point configurations such that the feasibility of the numerical problem of inverting the evaluation map (and thus determining $s$ from $s(x_1),\ldots,s(x_N)$) is good. See for example Chapter~9 in \cite{L} for more background related to this. 

Let $\{P_k\}_{k=1}^\infty$ be a sequence such that for each $k$, $P_k\in X^N$ and $P_k$ is a Fekete configuration, i.e.
\begin{equation} \label{eq:FeketeDef} |D(P_k)|_\phi = \max_{X^N} |D|_\phi \end{equation}
An important theme has been to study the asymptotic behaviour of $P_k$ as $k\rightarrow \infty$. We will use $\delta^N$ to denote the function from $X^N$ to the space of probability measures on $X$ given by
$$ (x_1,\ldots,x_N) \mapsto \frac{1}{N}\sum_{i=1}^N \delta_{x_i}. $$
The probability measure $\delta^N(P_k)$ is the \emph{empirical measure} of $P_k$. The well known equidistribution property for Fekete points says that
\begin{equation}
\label{eq:EquiDistm1}
    \delta^N(P_k) \rightarrow \mu_{(L,\phi)}
\end{equation}
in the weak topology of measures on $X$, where $\mu_{(L,\phi)}$ is the equilibrium measure of $(L,\phi)$, in other words the (normalized) non-pluripolar Monge-Ampère measure of $\mathcal P(\phi)$, where $\mathcal P(\phi)$ is the natural projection of $\phi$ onto the set of semi-positive metrics on $L$. For $n\geq 2$ the convergence in \eqref{eq:EquiDistm1} was proved in \cite{BBWN}, settling a long standing conjecture (see also \cite{MO} for the closely related case of spherical polynomials on the sphere). 

Relaxing \eqref{eq:FeketeDef}, we get the following asymptotic condition on a sequence of point configurations $\{P_k\}_{k=1}^\infty$
\begin{equation}
\label{eq:AssFeketeDef}
    \lim_{k\rightarrow \infty} \frac{1}{kN} \log |D(P_k)|^2_\phi = \lim_{k\rightarrow \infty} \frac{1}{kN} \log \max_{X^N} |D|^2_\phi.
\end{equation}
Sequences $\{P_k\}_{k=1}^\infty$ satisfying \eqref{eq:AssFeketeDef} are said to be \emph{asymptotically Fekete} with respect to $(L,\phi)$ and this condition is in fact sufficient to guarantee the equidistribution property in \eqref{eq:EquiDistm1} (see \cite{BBWN}, Theorem~C and the discussion after it). Note that the limit in the right hand side of \eqref{eq:AssFeketeDef} exist and is $<+\infty$ (see for example \cite{BB}).

It will be convenient to extend the definition of Vandermonde determinants to objects of the form $(L,\phi) = (\lambda L',\lambda \phi')$ where $\lambda$ is a positive real number and $(L',\phi')$ is an ample line bundle with a continuous Hermitian metric. 
We may without loss of generality assume that there is no representation $L=\lambda'L''$ such that $\lambda'>\lambda$ and $L''$ is a line bundle. Letting $\left\lfloor k\lambda \right\rfloor$ be the largest integer less than or equal to $k\lambda$, we define the associated Vandermonde determinant $| D|^2_\phi$ as
$$ |\det(s_i(x_l))_{i,l}|^2e^{- 2\left\lfloor k\lambda \right\rfloor\sum_{l=1}^{N} \phi'(x_l)} $$
where $N=N_L(k)=\dim H^0(X,\left\lfloor k\lambda \right\rfloor L')$ and 
$\{s_1,\ldots,s_N\}$ is a basis of $H^0(X,\left\lfloor k\lambda \right\rfloor L')$. Note that the equidistribution property for Fekete points extends to this setting. 

\subsection{Mutual Asymptotic Fekete Configurations}
\label{subsec:MutualAsy}
In this paper we will consider $m$ pairs $(L_1,\phi_1),\ldots,(L_m,\phi_m)$ where each $L_j$ is (a positive real multiple of) an ample line bundle over $X$ and each $\phi_j$ is (a positive real multiple of) a continuous Hermitian metric on $L_j$. 
Our main result concerns the existence of point configurations that have good sampling properties for all the vector spaces 
$$ H^0(X,kL_1), \ldots, H^0(X,kL_m). $$
For each $j\in \{1,\ldots,m\}$ we will let $N_j=N_{kL_j}$ and $|D_j|^2_{\phi_j}$ be the Vandermonde determinant function on $X^{N_j}$ determined by $(L_j,\phi_j)$. 
%In particular, letting $N_j$ be the dimension of $H^0(X,L_j)$ and fixing a basis of $H^0(X,L_j)$ for each $j$, we get a Vandermonde type determinant function $D_j:X^{N_j}\rightarrow \mathbb R$. 

Let $N=\max_j N_j$. For $P=(x_1,\ldots,x_N)\in X^N$ we will let $|D_j(P)|^2_\phi$ denote the quantity
$$ |D_j(P)|^2_\phi = |D_j(x_1,\ldots,x_{N_j})|^2_\phi $$
which only depend on the first $N_j$ entries of $P$. Assume 
\begin{equation}
    L_1^n=\ldots=L_m^n.
    \label{eq:VolumeNormIntro}
\end{equation} 
Then the the growth of the quantities $N_1,\ldots,N_m$ as $k\rightarrow \infty$ are comparable. More precisely, $\lim_{k\rightarrow \infty} N_j/N = 1$ for all $j\in\{1,\ldots,m\}$. We will ask the following question:
\begin{question}
\label{qu:MutualAssSeq}
Does there exist a sequence $\{P_k\}_{k=1}^\infty$ which is asymptotically Fekete with respect to $(L_j,\phi_j)$ for all $j\in \{1,\ldots,m\}$, i.e.
\begin{equation}
    \label{eq:MutualAsyFekete}
\lim_{k\rightarrow \infty} \frac{1}{kN} \log |D_j(P_k)|^2_\phi = \lim_{k\rightarrow \infty} \frac{1}{kN} \log \max_{X^N} |D_j|^2_\phi. 
\end{equation}
for all $j\in \{1,\ldots,m\}$?
\end{question}
%In particular, it is natural to ask this question when the dimension of the vector spaces $H^0(X,kL_1), \ldots, H^0(X,kL_m)$ grows in a comparable way. 
To accommodate collections of line bundles that don't satisfy \eqref{eq:VolumeNormIntro}, we will normalize their volumes. For each $j\in\{1,\ldots,m\}$, let $c_j=(L_j^n)^{1/n}$. We will write that a sequence $\{P_k\}_{k=1}^\infty$ is a \emph{mutual asymptotic Fekete sequence} for $(L_1,\phi_1),\ldots,(L_m,\phi_m)$ if it satisfies \eqref{eq:MutualAsyFekete} for the normalized pair
$\left(\frac{1}{c_j}L_j,\frac{1}{c_j}\phi_j\right)$, $j\in \{1,\ldots,m\}$.
%Given the interpretation of Fekete points as good sampling configurations, a mutual asymptotic Fekete sequence is a sequence with good sampling properties for all vector spaces
%$$ H^0(X,kL_1), \ldots, H^0(X,kL_m). $$

The equidistribution property \eqref{eq:EquiDistm1} gives a necessary condition for existence of a mutual asymptotic Fekete sequence, namely:
\begin{lemma}
\label{lem:NesCond}
Assume the pairs $(L_1,\phi_1),\ldots,( L_m,\phi_m)$ admits a mutual sequence of asymptotic Fekete configurations. Then 
\begin{equation} 
\label{eq:NesCond}
\mu_{(L_1,\phi_1)} = \ldots = \mu_{(L_m,\phi_m)}. 
\end{equation}
\end{lemma}
We conjecture that this condition is also sufficient:
\begin{conjecture}
\label{con:AsyFekete}
Let $(L_1,\phi_1),\ldots,(L_m,\phi_m)$ be a collection of pairs as above. Then $(L_1,\phi_1),\ldots,(L_m,\phi_m)$
admits a mutual asymptotic Fekete sequence if and only if \eqref{eq:NesCond} holds.
\end{conjecture}
%The significance of Conjecture~\ref{con:AsyFekete} is described in Theorem~\ref{thm:FeketeConv} and Theorem~\ref{thm:PtProc} below. 
Our main result confirms this conjecture in the toric setting, provided that $\sum_{j=1}^m \phi_j$ is regular enough. 
%Our first result confirms this conjecture in the toric setting, provided a certain regularity condition on the equilibrium measure $\mu_{(L_1,\phi_1)}$ holds. 
Recall that a toric manifold is a complex manifold $X$ of dimension $n$, together with an action by the $n$-dimensional complex torus $(\mathbb C^*)^n$ which admit a dense orbit in $X$. A Hermitian line bundle $(L,\phi)$ over $X$ is toric if the action on $X$ lifts to an action on $L$ and the Hermitian metric $\phi$ is invariant under the action of $(S^1)^n \subset (\mathbb C^*)^n$. 
\begin{theorem}
    \label{thm:ToricAsyFekete}
    Assume $(X,\omega)$ and $(L_1,\phi_1),\ldots,(L_m,\phi_m)$ are toric and 
    $\phi_1,\ldots,\phi_m$ are smooth or (more generally) the Monge-Ampère measure of $\mathcal P(\sum_{j=1}^m \phi_j)$ is absolutely continuous with bounded density.
    %$\mu_{(L_j,\phi_j)}$ satisfies $M_1$ for some $j\in \{1,\ldots,m\}$. 
    Then
    $$(L_1,\phi_1),\ldots,(L_m,\phi_m)$$ admits a mutual asymptotic Fekete sequence if and only if 
    \begin{equation} \label{eq:SameEqMeasures} \mu_{(L_1,\phi_1)} = \ldots = \mu_{(L_m,\phi_m)}. \end{equation}
\end{theorem}

\subsection{$\mathcal L$-functionals and coupled equilibrium energy}
In the general case, (when $(X,\omega)$ and $(L_1,\phi_1),\ldots, (L_m,\phi_m)$ are not necessarily toric) it turns out that the statement in Theorem~\ref{thm:ToricAsyFekete} is equivalent to a statement regarding the asymptotics of a sequence of $\mathcal L$-type functionals. To define these functionals, we first incorporate a volume normalization in the setup. 
%We start with $m$ pairs where each $L_j$ is (a positive real multiple of) an ample line bundle over $X$ and each $\phi_j$ is (a positive real multiple of) a continuous Hermitian metric. 
%As above, for each $j\in\{1,\ldots,m\}$, let $c_j=(L_j^n)^{1/n}$ and 
For each $j\in\{1,\ldots,m\}$, let $|\hat D_j|^2_{\phi_j}$ be the Vandermonde determinant associated to the volume normalized Hermitian line bundle $(\frac{1}{c_j}L_j,\frac{1}{c_j}\phi_j)$. We define the associated Vandermonde product as 
\begin{equation} \label{eq:VandermondeProduct} \prod_{j=1}^m \left(|\hat D_j|^2_{\phi_j}\right)^{c_j}. \end{equation}
Letting $\hat N_j=N_{\frac{k}{c_j}L_j}$ and $\hat N=\max_j \hat N_j$, we get that \eqref{eq:VandermondeProduct} defines a function on $X^{\hat N}$. Note that the $j$'th factor of \eqref{eq:VandermondeProduct} only depends on the $\hat N_j$ first coordinates of a point in $X^{\hat N}$. 

For simplicity we will assume $\int \omega^n = 1$. Given $(L_1,\ldots,L_m)$, we define the associated $\mathcal L$-functional at level $k$ as
$$ \mathcal L_k(\phi_1,\ldots,\phi_m) = \frac{1}{k\hat N}\log \int_{X^{\hat N}}\prod_{j=1}^m \left(\left|\hat D_j\right|^2_{\phi_j}\right)^{c_j}  \left(\omega^n\right)^{\otimes \hat N} $$
% \begin{remark}
% This is a generalization of Donaldson's $\mathcal L$-functional associated to an ample line bundle (see \cite{D} and \cite{BB}). However, note that even for $m=1$ this does not exactly follow \cite{D} and \cite{BB}, where $\mathcal L_k$ is defined as the normalized logarithm of a Gram determinant. By a computation (see for example Lemma~5.3 in \cite{BB}) the definitions are equivalent but it is not clear if a formula for $\mathcal L_k$ involving one (or several) Gram determinants exist in the case $m>1$.
% \end{remark}

Conjecturally, the limiting behaviour (as $k\rightarrow \infty$) of $\mathcal L_k$ is described by an energy functional akin to the classical equilibrium energy. More precisely, for each $j$ and continuous metric $\phi_j$ on $L$, let $\mathcal P_{L_j}(\phi_j)$ be the projection of $\phi_j$ onto the space of semi-positive metrics (plurisubharmonic weights) on $L$. Let $E_{L_j}$ be the (Aubin-Yau) Monge-Ampère energy and let $E^{eq}_{j}(\phi_j)$ be the \emph{equilibrium energy} defined by
$$E^{eq}_{j}(\phi_j) = \frac{E_{L_j}(\mathcal P_{\theta_i}(\phi_j))}{L_j^n}.$$ 
Note that the definition of $E_{L_j}$ depends on a choice of reference metric $\phi_j^0$ on $L_j$. For each $j$, we will make this choice such that the basis used in the definition of $|\hat D_j|^2_{\phi_j}$ is orthogonal with respect to the inner product induced by $\phi_j^0$ and the measure $\omega^n$. 
With this choice, Theorem~A in \cite{BB} takes the following form:
\begin{equation} \label{eq:ChebDiam} \lim_{k\rightarrow \infty} -\frac{1}{k\hat N_j}\log \max_{X^{\hat N_j}} |\hat D_j|^2_{\phi_j} = E^{eq}_j(\phi_j). \end{equation}

For a continuous metric $\phi$ on $\sum_{j=1}^m L_j$, let  $f_\phi$ be the functional given by
\begin{equation} 
    f_\phi(\phi_1,\ldots,\phi_m) = \sum_{j=1}^m E^{eq}_j(\phi_j) - \sup_X \left(\sum_{j=1}^m \phi_j - \phi\right). 
    \label{eq:f}
\end{equation}
We define the \emph{coupled equilibrium energy} $F$ as the functional on the space of continuous metrics on $\sum L_j$ given by
    \begin{equation} F(\phi) = \sup f_\phi \label{eq:EqEnergyDef} \end{equation}
where the supremum is taken over all continuous metrics $\phi_1,\ldots, \phi_m$ on $L_1,\ldots,L_m$. 
%This functional will play a similar role as the equilibrium energy in \cite{BBWN}. 
If $m=1$, it can be proved that
$ F=E^{eq}_1$ on 
$$ \{\phi_1:\sup_X (\phi_1-\phi)=0\}.$$
For $m>1$, $F$ can be thought of as (minus) the infimal convolution of the collection of convex functionals $-E_1^{eq}, \ldots, -E_m^{eq}$
(see \cite{H21a} for details). 
%Moreover, solutions to \eqref{eq:WeakEqMeasure1}-\eqref{eq:WeakEqMeasure2} maximize \eqref{eq:EqEnergyDef} (Lemma~1 in \cite{H21a}). In other words, if $\psi_1,\ldots,\psi_m$ is a solution to \eqref{eq:WeakEqMeasure1}-\eqref{eq:WeakEqMeasure2}, then 
%\begin{equation} \label{eq:EnergyAndEqPotentials} F(\phi) = f_\phi(\psi_1,\ldots,\psi_m) = \sum_{j=1}^m E^{eq}_j(\psi_j). \end{equation}

% The significance of F is the following differentiability property:
% \begin{theorem}[\cite{H21a}]
% \label{thm:Differentiability}
% The functional $F$ is Gateaux differentiable and its Gateaux derivative at $\phi$ is given by $\mu_{(L_1,\ldots,L_m,\sum_{j=1}^m \phi_j)}$. In other words, for any $v\in C(X)$, the function on $\mathbb R$ given by
% $$ g(t) = F(\phi+tv) $$
% is differentiable at $t=0$ and 
% $$ g'(0) = \int_X v \mu_{(L_1,\ldots,L_m,\sum_{j=1}^m \phi_j)}. $$
% \end{theorem}

The state the main theorem of this section we also need the following definition from \cite{H21a}:
\begin{definition}
\label{def:EqMeasure}
Let $L_1,\ldots,L_m$ be a collection of line bundles over $X$ and $\phi$ a continuous metric on $\sum_{j=1}^m L_j$. A collection of equilibrium potentials $(\psi_1,\ldots,\psi_m)$ is a collection of semi-positive metrics on $L_1,\ldots,L_m$ such that
\begin{equation} \label{eq:WeakEqMeasure1} (\idd \psi_1)^n = \ldots = (\idd \psi_m)^n \end{equation}
in the sense of non-pluripolar products and 
\begin{equation} \label{eq:WeakEqMeasure2} \sum_{j=1}^m \psi_j \leq \phi \end{equation}
with equality on the support of $(\idd \psi_1)^n$.
\end{definition}

In \cite{H21a} it is shown that \eqref{eq:WeakEqMeasure1}-\eqref{eq:WeakEqMeasure2} admits a solution, unique up to additive constants, for any ample line bundles $L_1,\ldots,L_m$ and continuous metric $\phi$ on $\sum_{j=1}^m L_j$. Moreover, the following variational property from \cite{H21a} will be important in Section~\ref{sec:LFunc}:
If $\psi_1,\ldots,\psi_m$ is a solution to \eqref{eq:WeakEqMeasure1}-\eqref{eq:WeakEqMeasure2}, then $(\psi_1,\ldots,\psi_m)$ are maximizers in \eqref{eq:EqEnergyDef}. In other words,
\begin{equation} \label{eq:EnergyAndEqPotentials} F(\phi) = f_\phi(\psi_1,\ldots,\psi_m) = \sum_{j=1}^m E^{eq}_j(\psi_j). \end{equation}

%Conjecturally, $F$ encodes the asymptotics of $\mathcal L_k$. More precisely: 
\begin{theorem}
    \label{thm:EqCond}
    Let $(L_1,\phi_1),\ldots,(L_m,\phi_m)$ be pairs as above (not necessary toric). Let $\phi=\sum_{j=1}^m \phi_j$ and assume $\phi_1,\ldots,\phi_m$ are smooth or (more generally) the Monge-Ampère measure of $\mathcal P(\phi)$ is absolutely continuous with bounded density. Let $(\psi_1,\ldots,\psi_m)$ be a solution to \eqref{eq:WeakEqMeasure1}-\eqref{eq:WeakEqMeasure2}. Then the following two conditions are equivalent:
    \begin{itemize}
        \item The $\mathcal L^k$-functional at $(\phi_1,\ldots,\phi_m)$ converges to the coupled energy at $\sum \phi_j$, i.e. \begin{equation} \label{eq:LkLimit} \lim_{k\rightarrow \infty} \mathcal L_k(\phi_1,\ldots,\phi_m) =   F\left(\sum_{j=1}^m \phi_j\right)\end{equation}
        \item The pairs $(L_1,\psi_1), \ldots, (L_m,\psi_m)$ admits a mutual asymptotic Fekete sequence. 
    \end{itemize}
\end{theorem}
\begin{remark}
We make a remark on the regularity assumption in Theorem~\ref{thm:EqCond}. Its main purpose is to guarantee that the potentials $\psi_1,\ldots,\psi_m$ of the equilibrium measure are continuous. When $\phi$ is smooth or the the Monge-Ampère measure of $\mathcal P(\phi)$ is absolutely continuous with bounded density, any solution $(\psi_1,\ldots,\psi_m)$ to \eqref{eq:WeakEqMeasure1}-\eqref{eq:WeakEqMeasure2} is continuous and the equilibrium measure is absolutely continuous with bounded density (see Theorem~2 in \cite{H21a}). 
\end{remark}

\subsection{Products of Vandermonde Determinants}
Another aspect of this paper is point configurations that maximize the product of Vandermonde determinants in \eqref{eq:VandermondeProduct}, in other words configurations $P$ such that 
\begin{equation} \label{eq:Feketem} \prod_{j=1}^m \left(|\hat D_j(P)|^2_{\phi_j}\right)^{c_j} = \max_{X^{\hat N}} \prod_{j=1}^m \left(|\hat D_j|^2_{\phi_j}\right)^{c_j}. \end{equation}
For a sequence of point configurations $\{P_k\}_{k=1}^\infty$ we get a more general condition than \eqref{eq:Feketem} by considering its asymptotics:
\begin{eqnarray}
\label{eq:AsyFeketem}
\lim_{k\rightarrow \infty} \left(\frac{1}{k\hat N}\log \prod_{j=1}^m \left(|\hat D_j(P_K)|^2_{\phi_j}\right)^{c_j} - \frac{1}{k\hat N}\log \max_{X^{\hat N}} \prod_{j=1}^m \left(|\hat D_j|^2_{\phi_j}\right)^{c_j}\right) = 0.
\end{eqnarray}
The significance of this condition is the following observation, which also play a key role in the proof of Theorem~\ref{thm:EqCond}:
\begin{lemma}
\label{lem:GoodPtConfigurations}
Assume $(L_1,\phi_1,\ldots,L_m,\phi_m)$ satisfies the first bullet in Theorem~\ref{thm:EqCond} and $\phi_1,\ldots,\phi_m$ are smooth or (more generally) the Monge-Ampère measure of $\mathcal P(\sum_{j=1}^m \phi_j)$ is absolutely continuous with bounded density. Let $\{P_k\}_{k=1}^\infty$ be a sequences of point configurations on $X$. Let $\psi_1,\ldots,\psi_m$ be a solution to \eqref{eq:WeakEqMeasure1}-\eqref{eq:WeakEqMeasure2} for $\phi = \sum_{j=1}^m \phi_j$. Then $\{P_k\}_{k=1}^\infty$ is a mutual asymptotic Fekete sequence of $(L_1,\psi_1),\ldots,(L_m,\psi_m)$ if and only if \eqref{eq:AsyFeketem} holds. 
\end{lemma}

Lemma~\ref{lem:GoodPtConfigurations} has two implications. First of all, it gives a concrete way of constructing mutual asymptotic Fekete configurations (if they exist), namely by maximizing the product of the associated Vandermonde determinants. Secondly, picking another point of view, it gives a new equidistribution property for maximizers of this product. 

To state this equidistribution property we will refer to $(\idd \psi_1)^n$ for a solution to \eqref{eq:WeakEqMeasure1}-\eqref{eq:WeakEqMeasure2} as the equilibrium measure of $(L_1,\ldots,L_m,\phi)$ and we will use the notation $\mu_{(L_1,\ldots,L_m,\phi)}$ for it. 
\begin{theorem}
     \label{thm:FeketeConv}
     Assume $(L_1,\phi_1,\ldots,L_m,\phi_m)$ satisfies one (and hence both) of the conditions in Theorem~\ref{thm:EqCond} and $\phi_1,\ldots,\phi_m$ are smooth or (more generally) the Monge-Ampère measure of $\mathcal P(\sum_{j=1}^m \phi_j)$ is absolutely continuous with bounded density. Assume $\{P_k\}_{k=1}^\infty$ satisfies
     $$ \prod_{j=1}^m \left(|\hat D_j(P_K)|^2_{\phi_j}\right)^{c_j} = \max_{X^{\hat N}} \prod_{j=1}^m \left(|\hat D_j|^2_{\phi_j}\right)^{c_j} $$
     or, more generally, satisfies the asymptotic condition in \eqref{eq:AsyFeketem}. 
     Then the empirical measures of $\{P_k\}_{k=1}^\infty$ converge weakly to the equilibrium measure of $(L_1,\ldots,L_m,\sum_{j=1}^m\phi_j)$, in other words
     $$ \delta^{\hat N}(P_k) \rightarrow \mu_{(L_1,\ldots,L_m,\sum_{j=1}^m\phi_j)}.
     $$
     %, i.e.
    %$$ \delta^N(P_k) \rightarrow \mu^{eq}(L_1,\ldots,L_m,\sum_{j=1}^m\phi_j). $$
\end{theorem}
\begin{remark}
Note that by Theorem~\ref{thm:ToricAsyFekete}, the assumptions in Theorem~\ref{thm:FeketeConv} are satisfied when $(X,\omega)$ and $(L_1,\phi_1),\ldots,(L_m,\phi_m)$ are toric and $(\phi_1,\ldots,\phi_m)$ are smooth or (more generally) the Monge-Ampère measure of $\mathcal P(\sum_{j=1}^m \phi_j)$ is absolutely continuous with bounded density.
\end{remark}

\subsection{Organisation of the paper}
In Section~\ref{sec:LFunc}, we prove Theorem~\ref{thm:EqCond}. A lower bound for $\mathcal L_k$ follows in a straight forward manner from the case $m=1$ in \cite{BB}. The essential point in the proof of Theorem~\ref{thm:EqCond} is that the upper bound holds if and only if a mutual asymptotic Fekete sequence exist. This latter point will follow from Lemma~\ref{lem:GoodPtConfigurations}. 
At the end of Section~\ref{sec:LFunc} we also prove Theorem~\ref{thm:FeketeConv}.
%In Section~\ref{sec:FeketeConv} we prove Theorem~\ref{thm:FeketeConv}. Given Theorem~\ref{thm:Differentiability}, Theorem~\ref{thm:FeketeConv} follows from an argument controlling the derivative of a concave function, similarly as in \cite{BBWN}. 

In Section~\ref{sec:Toric}, which is the longest section of the paper, we prove Theorem~\ref{thm:ToricAsyFekete}.
%and Theorem~\ref{thm:ToricFeketeConv}. Both of these are 
This is done by establishing the first bullet of Theorem~\ref{thm:EqCond} 
%(or Conjecture~\ref{con:Lk}) 
in the toric setting (see Theorem~\ref{thm:ToricLConv} below). As mentioned above, a lower bound for $\lim_{k\rightarrow \infty} \mathcal L_k$ follows from the case $m=1$. For the upper bound, we proceed by first establishing an upper bound in terms of discrete optimal transport costs (Lemma~\ref{lem:LowerBoundCost}). 
This reduces the proof of Theorem~\ref{thm:ToricLConv} to finding a sequence of point configuration (this time on $\mathbb R^n$) which can be used to approximate a continuous optimal transport cost by discrete optimal transport costs. 
%This is reminiscent of the problem of finding mutual asymptotic Fekete sequences for a collection of Hermitian ample line bundles. 
To show that such a sequence exist, we establish a bound controlling the transport cost $C(\mu,\nu)$ with respect to the cost function $-\langle\cdot,\cdot\rangle$ under perturbations of the measures $\mu$ and $\nu$. The bound, which is given in Lemma~\ref{lem:CLessThanW1}, holds when one of the measures has finite first moment and one of the measures has compact support. It involves the Wasserstein 1-distance and an $L^\infty$-type distance related to optimal transport. Using this bound, it turns out that a satisfiactory sequence of point configurations can be constructed from any sequence of points $\{x_i\}_{i=1}^\infty$ such that $\frac{1}{N}\sum_{i=1}^N \delta_N(x_1,\ldots,x_N)$ approximates the push forward of the equilibrium measure to $\mathbb R^n$ in Wasserstein 1-distance. By general properties of the Wasserstein 1-distance such a sequence exist as long as the push forward of the equilibrium measure has finite first moment. This latter point follows from the fact that the equilibrium measure is absolutely continuous with bounded density (which follows from the regularity assumption in Theorem~\ref{thm:ToricAsyFekete}). 
%and Theorem~\ref{thm:ToricFeketeConv}). 
%The integrability assumption in Theorem~\ref{thm:ToricAsyFekete} and Theorem~\ref{thm:ToricFeketeConv} guarantees that the push forward to $\mathbb R^n$ of the equilibrium measure has finite first moment, and hence that the Wasserstein 1-distance is well-defined. 

%An important step in the proof is to find a sequence of point configuration (this time on $\mathbb R^n$) which can be used to approximate a continuous optimal transport cost by discrete optimal transport costs. This is reminiscent of the problem of mutual asymptotic Fekete sequences, but here, as will be explained in Section~\ref{sec:LowerBndOT}, it suffices to find a sequence 

To prove Lemma~\ref{lem:LowerBoundCost} we use the explicit bases for $H^0(X,kL_j), k=1,2,\ldots$ determined by the lattice points of the associated polytopes. In the case $m=1$, expanding the determinants in $\mathcal L_k$ and integrating over orbits of the compact torus replaces the determinant by a permanent, closely related to a discrete optimal transport cost. This is the argument used in \cite{BPP} and \cite{HPP}. In the case $m\geq 2$, things are complicated by a number of "cross terms". Controlling the sign and non-vanishing of these is the main technical difficulty of the proof and an important ingredient in the argument is a uniqueness property of optimal transport maps between generic discrete measures (Lemma~\ref{lem:UniqueOptMap}). 

\begin{remark}
The proof of Theorem~\ref{thm:ToricLConv} is a variant of the argument for the case when $m=1$ in \cite{BPP}. A feature of this variant, apart from the fact that it works when $m\geq 2$, is that it is not restricted to compactly supported measures. This latter part is crucial. 
% The upshot of this variant,  Because of this, it can be used to give an elementary proof of toric special cases of the main result in \cite{BB}, bypassing the use of Gram determinants and Bergman kernel estimates.
\end{remark}

\begin{remark}
The main difficulty in proving convergence of $\mathcal L_k$ (and thus Conjecture~\ref{con:AsyFekete}) in the general setting is that for $m\geq 2$ there is no apparent way of writing the integral in $\mathcal L_k$ as a Gram determinant. Indeed, if this integral could be written as the Gram determinant of a suitable space of holomorphic sections associated to $(L_1,\ldots,L_m)$, then it is likely that convergence can be proved using Bergman kernel estimates for this space, similarly as in \cite{BB}. Another approach, both to prove convergence of $\mathcal L_k$ and to prove 
%Conjecture~\ref{con:AsyFekete} - Conjecture~\ref{con:Lk}
Conjecture~\ref{con:AsyFekete} directly, would be to use the quantitative estimates on equidistribution of Fekete points in \cite{LO}, or ideas from numerical approaches to Fekete points (see for example \cite{BCLSV}). 
\end{remark}
% \begin{remark}\label{rem:WeightNotation}
% In this paper, given a Hermitian metric $|\cdot|$ on a holomorphic line bundle, we use $\phi=\{\phi_\alpha\}$ to denote the collection of functions determined by fixing a covering $\{U_\alpha\}$ of $X$ with local holomorphic frames $\{f_\alpha\}$ of $L$ and putting $\phi_\alpha = -\log|f_\alpha|$. Note that, even though $\phi$ is not a well defined function on $X$, the difference $\phi_\alpha-\phi_\alpha'$, if $\psi_\alpha' = -\log |s_\alpha|'$ where $|\cdot|'$ is another Hermitian metric on $L$, is a well-defined function on $X$. As usual, we will write $\idd \phi$ to denote the $(1,1)$-form locally defined by $ \idd \phi_\alpha $. This is independent of $\alpha$ since $\phi_\alpha-\phi_{\alpha'} = \log |s_\alpha/s_{\alpha'}|$ on $U_\alpha\cap U_{\alpha'}$ is pluriharmonic.
% \end{remark}
\subsection{Acknowledgements}
The author would like to thank Norm Levenberg and Robert Berman for valuable discussions related to this paper. Special thanks to Norm Levenberg for reading and providing comments on a draft, and for pointing out the first application of Lemma~\ref{lem:GoodPtConfigurations} mentioned above. The author also thanks the Knut and Alice Wallenberg Foundation and the Olle Engkvist Foundation for financial support while working on this topic. 

\section{Asymptotics of $\mathcal L_k$-functionals}
\label{sec:LFunc}
Here, as in the introduction, $(X,\omega)$ will be a Kähler manifold and
$$(L_1,\phi_1),\ldots,(L_m,\phi_m)$$ 
will be a collection where each $L_j$ is (a positive real multiple of) an ample line bundle over $X$ and each $\phi_j$ is (a positive real multiple of) a continuous Hermitian metric on $L_j$. The purpose of this section is to prove Theorem~\ref{thm:EqCond}. Before we do that we will provide a proof of Lemma~\ref{lem:NesCond}.
\begin{proof}[Proof of Lemma~\ref{lem:NesCond}]
Let $\{P_k\}_{k=1}^\infty$ be a sequence of point configuration which is asymptotically Fekete with respect to $(L_j,\phi_j)$ for all $j\in \{1,\ldots,m\}$. By \cite{BBWN} (more precisely,  Theorem~C and (0.10) directly beneath Theorem~C in \cite{BBWN}) 
$$ \lim_{k\rightarrow \infty} \delta^{\hat N}(P_k) = \mu_{(L_j,\phi_j)}$$ 
for all $j\in \{1,\ldots,m\}$. This implies \eqref{eq:NesCond}. 
\end{proof}

We now turn to the proof of Theorem~\ref{thm:EqCond}. We begin with the following lemma, relating the $L^2$-norm of the product \eqref{eq:VandermondeProduct} with its $\sup$-norm. 
\begin{lemma}
\label{lem:L2LInf}
Let $\phi_1,\ldots,\phi_m$ be continuous metrics on $L_1,\ldots,L_m$. Then
$$ \limsup_{k\rightarrow \infty}\mathcal L_k(\phi_1,\ldots,\phi_m) \leq  \limsup_{k\rightarrow \infty} -\frac{1}{k\hat N} \log \max_{X^{\hat N}}\prod_{j=1}^{m} \left|\hat D_j\right|^2_{\phi_j} $$
and 
$$ \liminf_{k\rightarrow \infty}\mathcal L_k(\phi_1,\ldots,\phi_m) \geq  \liminf_{k\rightarrow \infty} -\frac{1}{k\hat N} \log \max_{X^{\hat N}}\prod_{j=1}^{m} \left|\hat D_j\right|^2_{\phi_j}. $$
\end{lemma}
In the language of \cite{BB}, Lemma~\ref{lem:L2LInf} guarantees a \emph{Bernstein-Markov property} as long as the metrics are continuous and the measure is given by a volume form. 
%As in \cite{BB}, this follows from the mean value inequality and plurisubharmonicity of the product \eqref{eq:VandermondeProduct}. 
\begin{proof}[Proof of Lemma~\ref{lem:L2LInf}]
First of all, since $\int_X\omega^n=1$, it is immediate that 
\begin{eqnarray} \liminf_{k\rightarrow \infty} \mathcal L_k(\phi) & = &
 \liminf_{k\rightarrow \infty}-\frac{1}{k\hat N}\log \int_{X^{\hat N}} \prod_{i=1}^m |\hat D_j|^2_{\phi_j} \left(\omega^n\right)^{\otimes \hat N} \nonumber \\
  &\geq&  
 \liminf_{k\rightarrow \infty}-\frac{1}{k\hat N}\log \max_{X^{\hat N}} \prod_{i=1}^m \left|\hat D_j\right|^2_{\phi_j}. \nonumber 
 \end{eqnarray}
 As in \cite{BB}, the other inequality will follow from the mean value inequality and the fact that the norm of the determinant in \eqref{eq:VandermondeExplicit} is subharmonic. More precisely, let $\epsilon>0$ and $\{B_1,\ldots,B_p\}$ be a covering of $X$ by open balls such that $|\phi_j(x)-\phi_j(x')|<\epsilon$ for all $j\in \{1,\ldots,m\}$ whenever $x,x'$ are in one of these balls. Let $(x_1,\ldots,x_{\hat N})\in X^{\hat N}$ be a point where $\prod_{j=1}^m |\hat D_j|^2_{\phi_j}$ attains its maximum and pick $l_1,\ldots,l_{\hat N}\in \{1,\ldots,p\}$ such that $x_i\in B_{l_i}$ for all $i\in \{1,\ldots,\hat N\}$. Letting
 $$\rho = \min_{1\leq l\leq p} \int_{B_l} \omega^n $$
 we get, using \eqref{eq:VandermondeExplicit}, that
 \begin{eqnarray}
 \int_{X^{\hat N}} \prod_{j=1}^m |\hat D_j|^2_{\phi_j} (\omega^n)^{\otimes \hat N} & \geq & \int_{B_{l_1}\times\ldots\times B_{l_{\hat N}}} \prod_{i=1}^k |\hat D_j|^2_{\phi_j} (\omega^n)^{\otimes \hat N} \nonumber \\
 & \geq & \delta^{\hat N} e^{-k\hat N\epsilon} \max_{X^{\hat N}} \prod_{j=1}^m |\hat D_j|^2_{\phi_j} \nonumber 
 \end{eqnarray}
 It follows that
 \begin{eqnarray} \limsup_{k\rightarrow \infty} \mathcal L_k(\phi) & = &
 \limsup_{k\rightarrow \infty}-\frac{1}{k\hat N}\log \int_{X^{\hat N}} \prod_{i=1}^m |\hat D_j|^2_{\phi_j} \left(\omega^n\right)^{\otimes \hat N} \nonumber \\
  &\leq&  
 2\epsilon + \liminf_{k\rightarrow \infty}-\frac{1}{k\hat N}\log \max_{X^{\hat N}} \prod_{i=1}^m \left|\hat D_j\right|^2_{\phi_j} \nonumber 
 \end{eqnarray}
 for any $\epsilon>0$. This proves the theorem. 
\end{proof}

Using Lemma~\ref{lem:L2LInf} and \eqref{eq:ChebDiam}, we get the following asymptotic upper bounds of $\mathcal L_k$:
\begin{lemma}
\label{lem:LUpperBound}
For each $j\in \{1,\ldots,m\}$, let $\phi_j'$ be a continuous metric on $L_j$. Assume $\phi$ is a smooth metric on $\sum_{j=1}^m L_j$, or (more generally) a continuous metric such that the Monge-Ampère measure of $\mathcal P(\phi)$ is absolutely continuous with bounded density. Assume also that
$$ \sum_{j=1}^m \phi_j' \geq \phi.$$ Then
$$ \liminf_{k\rightarrow \infty}\mathcal L_k(\phi_1',\ldots,\phi_m') \geq F\left(\phi\right). $$
%\leq \sum_{j=1}^m E^{eq}_j(\phi_j')
\end{lemma}
\begin{proof}
Let $\psi_1,\ldots,\psi_m$ be a solution to \eqref{eq:WeakEqMeasure1}-\eqref{eq:WeakEqMeasure2} and note that 
\begin{equation} \label{eq:SolSumBnd} \sum_{j=1}^m \psi_j \leq \phi \leq \sum_{j=1}^m \phi_j'. \end{equation} 
By Theorem~2 in \cite{H21a}, $\psi_1,\ldots,\psi_m$ are continuous. We have
\begin{eqnarray}
 \liminf_{k\rightarrow \infty}\mathcal L_k(\phi_1',\ldots,\phi_m') & = & 
 \liminf_{k\rightarrow \infty}-\frac{1}{k\hat N}\log \int_{X^{\hat N}} \prod_{j=1}^m |\hat D_j|^2_{\phi_j}(\omega^n)^{\otimes \hat N} \nonumber \\
 & \geq & 
 \liminf_{k\rightarrow \infty}-\frac{1}{k\hat N}\log \max_{X^{\hat N}} \prod_{j=1}^m \left|\hat D_j\right|^2_{\phi_j} \nonumber \\
  & \geq & 
 \liminf_{k\rightarrow \infty}-\frac{1}{k\hat N}\log \max_{X^{\hat N}} \prod_{j=1}^m \left|\hat D_j\right|^2_{\psi_j} \nonumber \\
 & \geq & 
 \lim_{k\rightarrow \infty}-\frac{1}{k\hat N}\log\prod_{j=1}^m\max_{X^{\hat N_j}} \left| \hat D_j\right|^2_{\phi_j} \nonumber \\
 & = & \sum_{j=1}^m E^{eq}_j(\psi_j) \nonumber \\
 & = & F(\phi).
% & = & \sup_{(x_1,\ldots,x_N)\in X^N} \prod_{i=1}^k \left|D_j(x_1,\ldots, x_{N_i})\right|_{\phi_j} \nonumber \\
% & \leq & \prod_{i=1}^k \sup_{(x_1,\ldots,x_{N_i})\in X^{N_i}}  \left|D_j(x_1,\ldots, x_{N_i})\right|_{\phi_j}. \nonumber
\end{eqnarray}
Here, the inequality in the third row uses $\lim_{k\rightarrow \infty} \hat N_j/\hat N=1$ and \eqref{eq:SolSumBnd}. The convergence in the second to last row is given by applying \eqref{eq:ChebDiam} to each term in the sum and the equality in the last row is given by \eqref{eq:EnergyAndEqPotentials}. 
\end{proof}

The key point in the proof of Theorem~\ref{thm:EqCond} is that a mutual asymptotic Fekete sequence for $(L_1,\psi_1), \ldots, (L_m,\psi_m)$ can be used to bound the limit of $\mathcal L_k$ from above by $\sum_{j=1}^m E_j^{eq}(\psi_j) = F\left(\sum_{j=1}^m \phi_j\right)$. This will follow from Lemma~\ref{lem:GoodPtConfigurations}, which we will prove now. 
\begin{proof}[Proof of Lemma~\ref{lem:GoodPtConfigurations}]
Assume first that $\{P_k\}_{k=1}^\infty$ is a mutual asymptotic Fekete sequence for $(L_1,\phi_1),\ldots,(L_m,\phi_m)$. Then 
\begin{eqnarray} \limsup_{k\rightarrow \infty} -\frac{1}{k\hat N}\log\prod_{j=1}^m |\hat D_j(P_k)|^2_{\phi_j} & = & \limsup_{k\rightarrow \infty} -\frac{1}{k\hat N}\log\prod_{j=1}^m \max_{X^{\hat N_j}} |\hat D_j|^2_{\phi_j} \nonumber \\
& \leq & \limsup_{k\rightarrow \infty} -\frac{1}{k\hat N}\log \max_{X^{\hat N}} \prod_{j=1}^m |\hat D_j|^2_{\phi_j}, \nonumber 
\end{eqnarray}
hence $\{P_k\}_{k=1}^\infty$ satisfies \eqref{eq:AsyFeketem}. Conversely, assume $\{P_k\}_{k=1}^\infty$ satisfies \eqref{eq:AsyFeketem}. Using the assumption on $\mathcal L_k$ (the first bullet in Theorem~\ref{thm:EqCond}), together with $\sum \psi_j \leq \phi = \sum_{j=1}^m \phi_j $ and  Lemma~\ref{lem:L2LInf}, we get
\begin{eqnarray}
\limsup_{k\rightarrow \infty} -\frac{1}{k\hat N}\log \prod_{j=1}^m\left| \hat D_j(P_k)\right|^2_{\psi_j} & \leq & \limsup_{k\rightarrow \infty} -\frac{1}{k\hat N}\log \prod_{j=1}^m \left| \hat D_j(P_k)\right|^2_{\phi_j}
\nonumber \\
%\limsup_{k\rightarrow \infty} \sum_{j=1}^m -\frac{1}{kN}\log \left| D_j(P)\right|_{\psi_j} & \leq & \limsup_{k\rightarrow \infty} \sum_{j=1}^m -\frac{1}{kN}\log \left| D_j(P)\right|_{\phi_j}
%\nonumber \\
& = & \limsup_{k\rightarrow \infty} - \frac{1}{k\hat N}\log  \max_{X^{\hat N}}  \prod_{j=1}^m \left|\hat D_j\right|^2_{\phi_j} \nonumber \\
& = & \lim_{k\rightarrow \infty} \mathcal L_k(\phi_1,\ldots,\phi_m) \nonumber \\
& = & F(\phi) \nonumber \\
& = & \sum_{j=1}^m E^{eq}_{j}(\psi_j) \nonumber \\
& = &  \sum_{j=1}^{m}  \lim_{k\rightarrow \infty}- \frac{1}{k\hat N}\log \max_{X^{\hat N}}|\hat D_j|^2_{\psi_j} \nonumber \\
& = & \lim_{k\rightarrow \infty} - \frac{1}{k\hat N}\log \prod_{j=1}^m \max_{X^{\hat N}}|\hat D_j|^2_{\psi_j} \label{eq:BoundOnSum}
\end{eqnarray}
Trivially, we have
%\begin{equation} 
%\label{eq:TrivialSupIneq}
$$
\liminf_{k\rightarrow \infty} -\frac{1}{k\hat N}\log \left| \hat D_j(P_k)\right|^2_{\psi_j} \geq \lim_{k\rightarrow \infty} -\frac{1}{k\hat N}\log \max_{X^{\hat N_j}}\left| \hat D_j\right|^2_{\psi_j}
$$
%\end{equation}
for each $j\in \{1,\ldots,m\}$. Now, assume 
$$\limsup_{k\rightarrow \infty} -\frac{1}{k\hat N}\log \left| \hat D_j(P_k)\right|^2_{\psi_j} > \lim_{k\rightarrow \infty} -\frac{1}{k\hat N}\log \max_{X^{\hat N_j}}\left| \hat D_j\right|^2_{\psi_j}$$
for some $j\in\{1,\ldots,m\}$. This would imply 
%$$ \limsup_{k\rightarrow \infty} \sum_{j=1}^m -\frac{1}{kN}\log \left| D_j(P_k)\right|^2_{\psi_j} > \lim_{k\rightarrow \infty} \sum_{j=1}^{m} -\frac{1}{kN}\log \max|D_j|_{\psi_j} $$
$$ \limsup_{k\rightarrow \infty} -\frac{1}{k\hat N}\log \prod_{j=1}^m\left| \hat D_j(P_k)\right|^2_{\psi_j} > \lim_{k\rightarrow \infty} -\frac{1}{k\hat N}\log \prod_{j=1}^m\max_{X^{\hat N}}|\hat D_j|^2_{\psi_j} $$
contradicting \eqref{eq:BoundOnSum}. We conclude that 
$$ \lim_{k\rightarrow \infty} \frac{1}{k\hat N}\log \left|\hat  D_j(P_K)\right|^2_{\psi_j} = \lim_{k\rightarrow \infty} \frac{1}{k\hat N}\log \max_{X^{\hat N}}\left|\hat  D_j\right|^2_{\psi_j} $$
for each $j\in \{1,\ldots,m\}$, and hence that $\{P_k\}_{k=1}^\infty$ is a mutual asymptotic Fekete sequence for $(L_1,\psi_1), \ldots, (L_m,\psi_m)$.
\end{proof}
We can now prove Theorem~\ref{thm:EqCond}.
\begin{proof}[Proof of Theorem~\ref{thm:EqCond}]
Assume the first bullet in the theorem holds and, 
for each $k$, let $P_k\in X^{\hat N}$ be a configuration where $\prod_{j=1}^m \left|\hat D_j\right|^2_{\phi_j}$ attains its maximum. By Lemma~\ref{lem:GoodPtConfigurations}, $P_k\in X^{\hat N}$ is a mutual Fekete configuration for for $(L_1,\psi_1), \ldots, (L_m,\psi_m)$.

Conversely, assume the second bullet in theorem holds and let $\{P_k\}_{k=1}^\infty$ be a sequence of configurations which is asymptotically Fekete with respect to $(L_j,\psi_j)$ for all $j\in \{1,\ldots,m\}$. Then, using Lemma~\ref{lem:L2LInf}, 
\begin{eqnarray}
\limsup_{k\rightarrow \infty} \mathcal L_k(\phi_1,\ldots,\phi_m) & \leq & \limsup_{k\rightarrow \infty} -\frac{1}{k\hat N}\log \max_{X^{\hat N}}  \prod_{j=1}^m \left|\hat D_j\right|^2_{\phi_j} \nonumber \\
& \leq & \limsup_{k\rightarrow \infty} -\frac{1}{k\hat N}\log  \prod_{j=1}^m \left|\hat D_j(P_k) \right|_{\phi_j}^2\nonumber \\
& = & \limsup_{k\rightarrow \infty} -\sum_{j=1}^m \frac{1}{k\hat N}\log \left| \hat D_j(P_k) \right|_{\psi_j}^2 \nonumber \\
&  & +\lim_{k\rightarrow \infty} 2\int_X \left(\sum_{j=1}^m(\phi_j - \psi_j) \right)\delta^{\hat N}(P_k) \label{eq:PhiPsiDiscrepancy} \\
& = & \sum_{j=1}^m E^{eq}_j(\psi_j) \nonumber \\
& = & F(\phi) \nonumber 
\end{eqnarray}
where the integral in \eqref{eq:PhiPsiDiscrepancy} vanishes since $\delta^{\hat N}(P_k)\rightarrow \mu_{(L_1,\psi_1)}$ weakly and $\sum_{j=1}^m (\phi_j- \psi_j)$ is continuous and vanishes on the support of $\mu_{(L_1,\psi_1)}$. 
The reverse inequality follows from Lemma~\ref{lem:LUpperBound}. We conclude that the first bullet in in Theorem~\ref{thm:EqCond} holds.
\end{proof}

We end this section with the proof of Theorem~\ref{thm:FeketeConv}. 
\begin{proof}[Proof of Theorem~\ref{thm:FeketeConv}]
By Lemma~\ref{lem:GoodPtConfigurations}, $\{P_k\}_{k=1}^m$ is a mutual asymptotic Fekete sequence of $(L_1,\psi_1),\ldots,(L_m,\psi_m)$. In particular, it is asymptotically Fekete for $(L_1,\psi_1)$ and the theorem follows from \eqref{eq:EquiDistm1}. 
\end{proof}

\section{The Toric Setting}
\label{sec:Toric}
In this section we will focus on the toric setting and prove Theorem~\ref{thm:ToricAsyFekete}. 
%and Theorem~\ref{thm:ToricFeketeConv}. 
Assume that the data $(X,\omega)$ and $(L_1,\phi_1),\ldots,(L_m,\phi_m)$ are toric. In other words, there is an action by $(\bbC^*)^n$ on $X$ with a dense orbit, the action extends to the line bundles $L_1,\ldots,L_m$ over $X$ and $\phi_1, \ldots,\phi_m$ and $\omega$ are invariant under the the induced action by the compact torus $(S^1)^n$ on $\sum_{j=1}^mL_j$. We will identify the dense orbit in $X$ with $(\bbC^*)^n$ and write $(\bbC^*)^n\subset X$. 
The main point of this section is to prove the following theorem:
% Before we state the main theorem of this section we will introduce a regularity condition for measures on toric manifolds. We will be interested in those $(S^1)^n$-invariant measures $\mu$ on $X$ which satisfies, for every torus invariant prime divisor $D$,
% \begin{equation} \label{eq:IntegrabilityCond} \int \log |s_D| \mu > -\infty. \end{equation}
% Here $s_D$ is a section such that $D=\{s_D=0\}$ and $|\cdot|$ is any continuous metric on the line bundle associated to $D$. The significance of this property is that the push forward of $\mu$ under the (partially defined) log map $X\mapsto \mathbb R^n$ has finite first moment if and only if $\mu$ satisfies \eqref{eq:IntegrabilityCond} for every torus invariant prime divisor $D$ (see Section~\ref{sec:Toric} and Lemma~\ref{lem:Finite1stMoment}). Clearly, $\mu$ satisfies this condition if $\mu$ is absolutely continuous with bounded density. 
% \begin{theorem}
% \label{thm:ToricLConv}
% Assume, as explained above, that the data $(L_1,\phi_1,\ldots,L_m,\phi_m)$ is toric. Assume also that the equilibrium measure $\mu_{(L_1,\ldots,L_m,\sum_{j=1}^m \phi_j)}$ satisfies \eqref{eq:IntegrabilityCond} for all torus invariant divisors $D$ of $X$. Then both conditions in Theorem~\ref{thm:EqCond} holds. 
% \end{theorem}
\begin{theorem}
\label{thm:ToricLConv}
Assume, as explained above, that $(X,\omega)$ and the data $$(L_1,\phi_1,\ldots,L_m,\phi_m)$$ are toric. Assume also that $\phi_1,\ldots,\phi_m$ are smooth or (more generally) that the Monge-Ampère measure of $\mathcal P(\sum_{j=1}^m \phi_j)$ is absolutely continuous with bounded density. Then both conditions in Theorem~\ref{thm:EqCond} holds. 
\end{theorem}
The theorem will follow from a proposition verifying the first bullet in Theorem~\ref{thm:EqCond}. Since the second bullet in Theorem~\ref{thm:EqCond} is invariant under scaling of $L_1,\ldots,L_m$, we may without loss of generality assume that
$$ \Vol(L_1) = \ldots = \Vol(L_m) = 1 $$
and hence
$$ c_1 = \ldots = c_m = 1. $$
Note that this means $\hat D_j = D_j$ and $\hat N_j=N_j$ for all $j\in \{1,\ldots,m\}$ and $\hat N = N$.

Before we state the proposition that implies Theorem~\ref{thm:ToricLConv} we will introduce a regularity condition for measures on toric manifolds. Let $s$ be a section on a toric line bundle over $X$ which vanishes on all $(\mathbb C^*)^n$-invariant divisors of $X$ (equivalently, $s$ corresponds to a point in the interior of the associated polytope). Let $|\cdot|$ be a $(s^1)^n$-invariant metric on the line bundle of $s$. We will be interested in those measures $\mu$ on $X$ which does not charge pluripolar sets and for which $\log|s|$ is integrable, i.e. 
\begin{equation} \label{eq:IntegrabilityCond} \int \log |s| \mu > -\infty. \end{equation}

The significance of this property is that the push forward of $\mu$ under the (partially defined) log map $\pi:X\mapsto \mathbb R^n$ has finite first moment if and only if $\mu$ satisfies \eqref{eq:IntegrabilityCond}  (see Lemma~\ref{lem:Finite1stMoment} below). Clearly, $\mu$ satisfies \eqref{eq:IntegrabilityCond} if $\mu$ is absolutely continuous with bounded density. 
\begin{proposition}
\label{prop:ToricLConv}
Assume that the data $(L_1,\phi_1,\ldots,L_m,\phi_m)$ is toric, 
$$ \Vol(L_1) = \ldots = \Vol(L_m) = 1, $$
and that the equilibrium measure $\mu_{(L_1,\ldots,L_m,\sum_{j=1}^m \phi_j)}$ satisfies \eqref{eq:IntegrabilityCond}. Then the first bullet in Theorem~\ref{thm:EqCond} holds. 
\end{proposition}
Most of this section is dedicated to proving Proposition~\ref{prop:ToricLConv}. We will begin by recalling some well known facts about toric Kähler geometry (see for example \cite{BB} for a more detailed exposition). The main points are given by Lemma~\ref{lem:ToricLk} and Lemma~\ref{lem:ToricEnergy} below. 
Associated to each pair $(X,L_j)$ we get a rational polytope $P_j$ in $\bbR^n$. Let 
$$w = (w_1,\ldots,w_n)= x+\sqrt{-1}y = (x_1+\sqrt{-1}y_1,\ldots,x_n+\sqrt{-1}y_n) $$
be logarithmic coordinates on $(\bbC^*)^n$. Each $L_j$ admits a unique $(\bbC^*)^n$-invariant section defined over $(\mathbb C^*)^n\subset X$. Using this section to trivialize $L_j$ over $(\bbC^*)^n$ we get that each point $p\in \bbZ^n$ defines a holomorphic function $f_p:(\bbC^*)^n\rightarrow \mathbb C$
$$ f_p(w) = e^{\langle w, p \rangle} $$ 
which extends to a global section $s_p$ of $kL_j$ if and only if $p\in kP_j\cap \bbZ^n$. Moreover, 
\begin{equation} \label{eq:basis} \{f_p: p\in kP_j\cap \bbZ^n\} \end{equation}
is a basis for $H^0(X,kL_j)$ and any $(S^1)^n$-invariant continuous metric $\phi_j$ on $L_j$ may be represented by a continuous function on $\bbR^n$ which, abusing notation, we will also denote $\phi_j$. This function satisfies the growth condition 
$$ \sup_X \left| \phi_j - h_{P_j}\right| < C $$
on $\mathbb R^n$ for some constant $C$ where $h_{P_j}$ is the support function of $P_j$
$$ h_{P_j}(x) = \sup_{p\in P_j} \langle x,p\rangle. $$

We will use $h_{P_j}$ as the reference when defining $E^{eq}_j$ and $|D_j|^2_{\phi_j}$. It is easy to verify that the elements in \eqref{eq:basis} are pairwise orthogonal with respect to the inner product induced by any $(S^1)^n$-invariant metric on $L_j$. It follows that
\begin{equation} \label{eq:ONBasis} \left\{\frac{f_p}{\sqrt{\int_{(\mathbb C^*)^n}|f_p|^2e^{-2h_{P_j}}\omega^n}}: p\in kP_j\cap \bbZ^n\right\} \end{equation}
is an orthonormal basis with respect to the inner product induced by $h_{P_j}$.

For each $j$ and $k$, enumerating the points in $kP_j\cap \bbZ^n$ 
$$ kP_j\cap \bbZ^n = \{ p_1^j,\ldots,p_{N_j}^j\}. $$
gives
\begin{lemma}[See section 7.1 in \cite{BPP} for the case $m=1$]
\label{lem:ToricLk}
Let $p^j$ and the coordinates $x+\sqrt{-1}y$ be defined as above. Assume $ L_1^n = \ldots = L_m^n = 1 $. Then $\mathcal L_k$ takes the following form:
\begin{equation}
    \mathcal L_k(\phi_1,\ldots,\phi_m) =  -\frac{1}{kN} \log \int_{((\bbC^*)^n)^{N}} \prod_{j=1}^m \left|D_j\right|^2_{\phi_j}(\omega^n)^{\otimes N} + o_k(1)
    \label{eq:ToricLk}
\end{equation}
where 
$$ \left|D_j\right|^2_{\phi_j} = \left|\det\left(e^{\langle x_i+\sqrt{-1}y_i,p^j_l \rangle} \right)_{1\leq i,l\leq N_j}\right|^2e^{-2k\sum_{i=1}^{N_j} \phi_j(x_i)}. $$
and $o_k(1)\rightarrow 0$ as $k\rightarrow \infty$.
% \begin{eqnarray}
%      & \mathcal L_k(\phi_1,\ldots,\phi_m) & \nonumber \\
%      & = & \nonumber \\
%      & -\frac{1}{kN} \log \int_{((\bbC^*)^n)^N} \prod_{j=1}^m \left|\det\left(e^{\langle x_i+\sqrt{-1}y_i,p_l \rangle} \right)_{1\leq i,l\leq \hat N_j}\right|^2e^{-k\sum_{i} \phi(x_i)} \left(\omega^n\right)^{\otimes N} & \label{eq:ToricLk}
% \end{eqnarray}
\end{lemma}
\begin{proof}
Using \eqref{eq:ONBasis} as the basis in the definition of $|D_j|^2_{\phi_j}$ it is easy to verify that
\begin{eqnarray} \mathcal L_k(\phi_1,\ldots,\phi_m) & + & \frac{1}{kN} \log \int_{((\bbC^*)^n)^{N}} \prod_{j=1}^m \left|D_j\right|^2_{\phi_j}(\omega^n)^{\otimes N} \nonumber \\
& = & \frac{1}{kN}\log \prod_{j=1}^N \prod_{p\in kP_j\cap \mathbb Z} \int_{(\mathbb C^*)^n} |f_p|^2e^{-2kh_{P_j}} \omega^n. \label{eq:errortermLk} 
\end{eqnarray}
I thus suffices to show that \eqref{eq:errortermLk} vanishes as $k\rightarrow \infty$. 
Since 
$$ |f_p(w)|^2e^{-2kh_{P_j}(x)} = e^{2\langle x,p\rangle-2kh_P(x)} \leq 1 $$
we get that the limit of \eqref{eq:errortermLk} is non-positive. Moreover, let $R>0$ be a constant such that for each $j\in \{1,\ldots,m\}$, $P_j$ is contained in the unit ball of radius $R$ centered at the origin. For $\epsilon > 0$, let $B_\epsilon$ be the unit ball in $\mathbb R^n$ of radius $\epsilon$ centered at the origin and $\delta = \int_{B_\epsilon} \pi_\#(\omega^n)$. Then
\begin{eqnarray} \eqref{eq:errortermLk} & \geq & \frac{1}{kN} \log \prod_{j=1}^m \prod_{p\in kP_j\cap \mathbb Z} \int_{B_\epsilon} e^{2\langle x,p\rangle - 2kh_{P_j}(x)} \pi_\#(\omega^n) \nonumber \\
& \geq & \frac{1}{kN}\log \prod_{j=1}^m \prod_{p\in kP_j\cap \mathbb Z} e^{-k4R\epsilon}\delta \nonumber \\
& = & \frac{N_j}{N}\sum_{j=1}^m \left(-4R\epsilon+\frac{\log(\delta)}{k}\right)
\end{eqnarray}
which converges to $-4mR\epsilon$ as $k\rightarrow \infty$. Letting $\epsilon\rightarrow 0$ in this inequality proves that $\eqref{eq:errortermLk}\rightarrow 0$ as $k\rightarrow \infty$.
\end{proof}

Similarly as $\mathcal L^k$, $F$ has an alternative form in the toric setting. To write it, recall that the Legendre Transform of a lower semi-continuous function $\phi$ on $\mathbb R^n$ is given by
$$ \phi^*(p) = \sup_{x\in \bbR^n} \langle x,p \rangle -\phi(x). $$ 
%Moreover, if $\phi_j$ represent an $(S^1)^n$-invariant metric on $L_j$, then the projection operator $\mathcal P$ takes the following form
%$$ \mathcal P(\phi_j) = (\phi_j^*)^*. $$ 
If we use $\nu_{P_j}$ to denote the Lebesgue measure on $P_j$ then the equilibrium energy $E_{L_j}^{eq}(\phi_j)$ (with respect to the reference metric $h_{P_j}$) is, up to a volume normalizing constant, given by integration of $\phi_j^*$ against $\nu$, i.e.
\begin{equation} \label{eq:ToricMAEnergy} \hat E_{L_j}(\phi) = -\frac{\int_{P_j} \phi^* \nu_{P_j}}{L_j^n}. \end{equation}
\begin{lemma}
\label{lem:ToricEnergy}
Assume $L_1=\ldots=L_m=1$ and let $(\psi_1,\ldots,\psi_m)$ be the unique solution to \eqref{eq:WeakEqMeasure1}-\eqref{eq:WeakEqMeasure2} for $\phi = \sum_{j=1}^m \phi_j$. Then $\psi_1,\ldots,\psi_m$ are $(S^1)^n$-invariant and 
\begin{equation} \label{eq:ToricEnergy} F(\phi) = \sum_{j=1}^m -\int_{P_j} \psi_j^* \nu_{P_j}. \end{equation}
\end{lemma}
\begin{proof}
To see that $\psi_1,\ldots,\psi_m$ are $(S^1)^n$-invariant, note \eqref{eq:WeakEqMeasure1}-\eqref{eq:WeakEqMeasure2} are $(S^1)^n$-invariant. By uniqueness of solutions (Theorem~1 in \cite{H21a}), any solution has to be $(S^1)^n$-invariant. This proves the first part of the lemma. The second part follows from the first part, \eqref{eq:EnergyAndEqPotentials} and \eqref{eq:ToricMAEnergy}.
% Let $\mu_{Leb}$ be the Lebesgue measure on Consider the map $\proj$ that maps any continuous function $\phi$ on $\mathbb (C^*)^n$ to its fiberwise average with respect to $\pi:\mathbb (C^*)^n \rightarrow \bbR^n$
% $$ (\proj \phi)(x) = \int_{\pi^-1} \phi \mu_{Leg}. $$
% Consider the space 
% $$ B = \left\{(\psi_1,\ldots,\psi_m): \proj \psi_j = \proj \phi_j \text{ for all $j$.}\right\} $$
% Note that $(\phi_1,\ldots,\phi_m)\in B$ and 
% \begin{equation} 
% \label{eq:MinOnFiber} \inf_{A_\phi\cap B} \sum E_j(\psi_j) = \sum E_j(\phi_j). 
% \end{equation}
% It is easy to verify that $B$ is a convex space. By strict convexity of $E_j$ for each $j$ we get that $(\phi_1,\ldots,\phi_m)$ is the unique element in $B$ satisfying \eqref{eq:MinOnFiber}. Since $E_j$ is invariant under the action of $(S^1)^n$ for each $j$ we get that $(\phi_1,\ldots,\phi_m)$ has to be $(S^1)^n$-invariant. This proves the first part of the lemma. The second part follows from the first part and \eqref{eq:ToricMAEnergy}.
\end{proof}

Before we move on we will prove the following lemma, explaining the significance of the integrability condition in Proposition~\ref{prop:ToricLConv}.
\begin{lemma}
\label{lem:Finite1stMoment}
Assume $\mu$ is a probability measure on $X$ that does not charge pluripolar sets. Let $s$ be a section of a toric line bundle over $X$ which vanishes on all $(\mathbb C^*)^n$-invariant divisors of $X$ (equivalently, $s$ corresponds to a point in the interior of the associated polytope). Let $|\cdot|$ be a $(S^1)^n$-invariant metric on the line bundle of $s$. Then $\log|s|$ is integrable with respect to $\mu$ if and only if the push forward measure $\pi_\# \mu$ on $\mathbb R^n$ has finite first moment.
\end{lemma}
\begin{proof}
Let $L$ be the line bundle of $s$ and $P$ its polytope. Note that 
$$ |s|=e^{\langle x,p\rangle-\phi} $$
where $p$ is in the interior of $P$ and $\phi$ is a weight satisfying $|\phi-h_P]<C$ for some constant $C$. First of all, since $P$ is bounded, 
$$ -\log|s| = \phi(x)-\langle x, p\rangle \leq h_P(x)-\langle x, p\rangle+C \leq R|x|+C $$
for some $R>0$. Moreover, since $p$ is in the interior of $P$, 
$$ -\log|s| = \phi(x)- \langle x, p\rangle \geq h_P(x)-\langle x, p\rangle-C \geq r|x|-C$$
for some $r>0$ determined by the distance of $p$ to the boundary of $P$. It follows that
$$ r\int |x| \pi_\#(\omega^n) \leq -\int \log|s|\omega^n \leq R\int |x| \pi_\#(\omega^n). $$
This proves the lemma. 
% Let $\{D^1,\ldots,D^p\}$ be the torus-invariant prime divisors of $X$. For each $i\in \{1,\ldots,p\}$, let $d^i=(d^i_1,\ldots,d^i_n)\in \mathbb Z^n$ be the generator of the ray associated to $D_i$ in the fan of $X$. It follows that
% $$ s_D $$ is a section of $-K_X$ and using the standard trivialization of $-K_X$ on $(\bbC^*)n$ we get that $s_D$ may be represented by the function
% \begin{equation} \label{eq:DefiningSections} f_D(w) e^{\langle w,d_i\rangle}. \end{equation}
% on $(\bbC^*)n$. Moreover, by standard properties of compact toric manifolds, the convex hull of $\{d_1,\ldots,d_p\}$ in $\mathbb R^n$ contains 0 in its interior. It follows that there are constants $0<r<R$ such that 
% $$ r|x| < \sup_i \langle d_i,x \rangle < R|x|. $$
% In particular, $\int_{\bbR^n} |x| \pi_\# \mu<\infty$ if and only if 
% $$ \int_{\bbR^n} \sup_i \langle d_i,x \rangle \pi_\# \mu\mu <\infty. $$
% By \eqref{eq:DefiningSections}, $\langle d_i,x \rangle = -\log |f_D(w)|$. We get
% \begin{eqnarray} 
% \int_{\bbR^n} \sup_i \langle d_i,x \rangle \pi_\# \mu & = & -\int_{X} \inf_i \log |f_D| \mu \nonumber \\
% & = & -\int_{X} \inf_i \log |s_D|. \mu \nonumber
% \end{eqnarray}
% Since $|s_{D_i}|$ is bounded on $X$ for each $i$, we get that the right hand side of this is finite if and only if $\int_X \log |s_{D_i}(w)| \mu>-\infty$ for each $i$. 
\end{proof}

% \subsection{Variational Framework in the Toric Setting}
% Assume $\psi_1,\ldots,\psi_m$ be the associated functions on $\mathbb R^n$. Then the projection operator takes the following form $\phi \mapsto (\phi^*)^*$ where $*$ denotes Legendre transform. And the composition of Monge-Ampère energy $E_{\theta_i}$ with the projection has a formula in terms of the Legendre transform of $\phi$. Let $u_i=\phi_j-\psi_i$. Then (see for example \cite{BB}, Lemma ...)
% \begin{eqnarray} 
% E_{L_j}(u) & = & -\int_P (\psi_j+u^*) d\nu_j+\int_P \psi_j^* \nu_j \nonumber \\
% & = & -\int_P \phi_j^* d\nu_j+\int_P \psi_j^* \nu_j.
% \end{eqnarray}
% Moreover, the Pluricomplex Energy admits a formula in terms of the optimal transport cost from $\mu$ to the (unit volume) normalized Lebesgue measure $\nu_P$ on $P$. 

\subsection{A lower bound in terms of discrete optimal transport costs}\label{sec:LowerBndOT}
For two probability measure $\mu$ and $\nu$ on $\mathbb R^n$ we will let $C(\mu,\nu)$ denote the optimal transport cost when transporting $\mu$ to $\nu$ with respect to the cost function $c(x,y) = -\langle x,y \rangle$. In other words
$$ C(\mu,\nu) := \min_\gamma \int_{\mathbb R^n\times \mathbb R^n} c(x,y) \gamma $$
where the minimum is taken over all probability measures $\gamma$ on $\bbR^n\times \bbR^n$ whose marginals are given by $\mu$ and $\nu$ respectively. 
% such that if $\pi_1:\bbR^n\times \bbR^n\rightarrow \bbR^n$ and $\pi_2:\bbR^n\times \bbR^n\rightarrow \bbR^n$ are projection onto the first and second factor of $\bbR^n\times \bbR^n$, respectively, then $(\pi_1)\# \gamma = \mu$ and $(\pi_2)\# \gamma = \nu$. 

Assuming sufficient regularity of $\mu$ and $\nu$, for example compact support $P$ of $\nu$ and finite first moment of $\mu$, $C(\mu,\nu)$ coincides with its Kantorovich dual and the Kantorovich dual admits a maximizer, i.e.
\begin{equation}
\label{eq:KantCost}
C_{dual}(\mu,\nu) = \max_\phi \left(-\int_{\mathbb R^n} \phi \mu -\int_P \phi^* \nu \right)
\end{equation}
where the maximum is taken over all continuous functions $\phi:\mathbb R^n\rightarrow \mathbb R\cup\{\infty\}$ (\cite{V}, see the proof of Theorem~\ref{thm:ToricLConv} for details). Moreover, if we let $\mu_{eq}$ denote the push forward of the equilibrium measure
$$ \mu_{eq}  = \pi_\#(\mu_{(L_1,\ldots,L_m,\sum_{j=1}^m\phi_j)}) = \pi_\#(\mu_{(L_j,\psi_j)})$$
and put $\mu=\mu_{eq}$ and $\nu=\nu_j$ then $\psi_j$ is a maximizer in \eqref{eq:KantCost}. Consequently, using Lemma~\ref{lem:ToricEnergy}, $F(\phi)$ can be written as 
\begin{equation}
    F(\phi) = -\sum_{j=1}^m C(\mu_{eq},\nu_j) + \int_{\mathbb R^n} \phi \mu_{eq}.
    \label{eq:ToricEnergyCostForm}
\end{equation}

We will also be interested in the special case when $\mu$ and $\nu$ are given by a normalized sum of a finite number of point masses, i.e. when 
\begin{equation} \label{eq:DiscreteMeasures} \mu=\frac{1}{N} \sum_{i=1}^N \delta_{x_i} \text{ and } \nu= \frac{1}{N} \sum_{i=1}^N \delta_{p_i} \end{equation}
for some $N>0$ and point sets $\{x_1,\ldots,x_N\}$ and $\{p_1,\ldots,p_N\}$ in $\bbR^n$. Then any permutations $\sigma$ of the set $\{1,\ldots,N\}$ determines a measure $\gamma$ satisfying the push forward relations above, namely
$$ \gamma_\sigma = \sum_{i=1}^N \delta_{(x_i,p_{\sigma(i)})}, $$
and
$$ \int_{\mathbb R^n\times \mathbb R^n} c(x,y) \gamma_{\sigma} = -\frac{1}{N} \sum_{i=1}^N \langle x_i,p_{\sigma(i)} \rangle $$
Moreover, it can be shown that it suffices to consider $\gamma$ on this form, i.e. when $\mu$ and $\nu$ are given by \eqref{eq:DiscreteMeasures}
\begin{equation} \label{eq:DiscreteCost} C(\mu,\nu) = \min_\sigma -\frac{1}{N} \sum_{i=1}^N \langle x_i,p_{\sigma(i)} \rangle \end{equation}
where the minimum is taken over all permutations $\sigma$ of $\{1,\ldots,N\}$. 

%\begin{remark}
%Recalling that the permanent of an $N\times N$ matrix %$(a_{ij})$ is given by the formula
%$$ \perm (a_{il}) = \sum_{\sigma\in S_N} \prod_{i=1}^N a_{i\sigma(i)}. $$
%The tropical semi ring is given by semi ring $(\bbR \cup \{+\infty\},\oplus,\otimes)$ where $\oplus$ and $\otimes$ are defined by $a\oplus b = \min\{a,b\}$ and $a\otimes b = a+b$. Consequently, the formula in the right hand side of \eqref{eq:DiscreteCost} can (up to a factor of $1/N$) be seen as the 'tropical permanent' of the matrix $(-\langle x_i,p_l \rangle)$. The fact that Vandermonde type determinants of sections of high powers of linebundles sometimes can be related to tropical permanents is a good heuristic explanation of the connection between these and optimal transport exhibited here as well as in \cite{HultgrenPerm} and \cite{BermanPerm}.
%\end{remark}

%\subsection{Lower Bound in Terms of Discrete Optimal Transport Costs}
For any vector $x=(x_1,\ldots,x_{N'})$ of length at least $N$, we will use $\delta^{N}(x)$ to the denote the measure $\frac{1}{N}\sum_{i=1}^N \delta_{x_i}$ which only depend of the first $N$ entries of $x$. The following lemma gives an asymptotic upper bound of $\mathcal L_k$ in terms of a sum of discrete optimal transport costs. 
\begin{lemma}
\label{lem:LowerBoundCost} Let $p^j$ and $C(\cdot,\cdot)$ be defined as above and assume $L_1^n=\ldots=L_m^n=1.$ 
Then, for any $x=(x_1,\ldots,x_{N})\in (\mathbb R^n)^{N}$
%$$ \liminf_{k\rightarrow \infty} \mathcal L_k(\phi) \geq \liminf_{k\rightarrow \infty} \sup_{(x_1,\ldots,x_N)\in (\bbR^n)^N} \sum_{j=1}^m \max_{\sigma\in S_{N_j}} \sum_{i=1}^N \left(\langle x_i,p^j_{\sigma_j(i)} \rangle  - \phi(x_i)\right). $$
\begin{equation} \label{eq:LowerBoundCost} \mathcal L_k(\phi_1,\ldots,\phi_m) 
\leq 
\sum_{j=1}^m - C\left(\delta^{N_j}(x), \delta^{N_j}\left(\frac{p^j}{k}\right)\right) - \int_{\bbR^n} \sum_{j=1}^m\phi_j \delta^{N}(x) +o_k(1) 
\end{equation}
where $o_k(1)$ is a term which vanishes as $k\rightarrow \infty$.
\end{lemma}
Before proving this we will explain how it implies Proposition~\ref{prop:ToricLConv}. 
By Lemma~\ref{lem:LUpperBound}, Lemma~\ref{lem:LowerBoundCost} and \eqref{eq:ToricEnergyCostForm}, Proposition~\ref{prop:ToricLConv} will follow if we show that there exist some sequence $\{x_i\}_{i=1}^\infty$ in $\mathbb R^n$ such that 
$$ C\left(\delta^{N^j}(x_1,\ldots,x_{N_j}),\delta^{N_j}\left(\frac{p^j}{k}\right)\right) \rightarrow C(\mu_{eq},\nu_j) $$
for all $j\in \{1,\ldots,m\}$ and 
$$ \int_{\mathbb R^n} \phi \delta^{N}(x_1,\ldots,x_{N}) \rightarrow \int_{\mathbb R^n} \phi \mu_{eq}. $$
Existence of such a sequence will follow from topological properties of the Wasserstein $p$-space, in other words the space of probability measures on $\mathbb R^n$ with finite $p$'th moment endowed with the Wasserstein $p$-metric. The Wasserstein $p$-metric $W_p$ is defined as the $p$'th root of the optimal transport cost of moving one measure to another with respect to the cost function $c(x,y) = |x-y|^p$. The cruzial property we will use is that for any probability measure $\mu$ of finite $p$'th moment, there exist a sequence of points $\{x_i\}_{i=1}^\infty$ in $\mathbb R^n$ such that the empirical measures $\frac{1}{N}\sum_1^N \delta_{x_i}$ converges in Wasserstein $p$-distance to $\mu$ as $N\rightarrow \infty$. In fact, sampling points independently according to $\mu$ will produce such a sequence with probability~1 (see Section~2.2.2 and Proposition~2.2.6 in \cite{PanaretosZemel}). 

In the case when $\mu_{eq}$ has finite second moment the argument is quite simple. First of all, we may rewrite the right hand side of \eqref{eq:ToricEnergyCostForm} as 
\begin{equation}
    \label{eq:EnergyFinite2ndMoment}
    \sum_{j=1}^m \left(W^2(\mu_{eq},\nu_j) - \int_{\mathbb R^n} |x^2| \mu_{eq} - \int_{P_j} |p|^2\nu_j\right) - \int_{\mathbb R^n} \phi \mu_{eq} 
\end{equation}
and similarly the right hand side of \eqref{eq:LowerBoundCost} as 
\begin{equation}
    \label{eq:DiscreteCostFinite2ndMoment}
    \sum_{j=1}^m \left(W^2(\delta^{N_j}(x),\delta^{N_j}(p^j)) - \int_{\mathbb R^n} |x^2| \delta^{N_j}(x) - \int_{P_j} |p|^2\delta^{N_j}\left(\frac{p^j}{k}\right)\right) - \int_{\mathbb R^n} \phi \delta^{N}(x). 
\end{equation}
It is well known that convergence in $W_p$ is equivalent to weak convergence of measures and convergence of $p$'th moments. The latter also implies convergence of the pairing with any function bounded by a multiple of $|x|^p$. For each $k$, picking $x\in (\mathbb R^n)^N$ as the first $N$ entries $x=(x_1,\ldots,x_N)$ of a sequence $\{x_i\}_{i=1}^\infty$ with the property above, the expression in \eqref{eq:DiscreteCostFinite2ndMoment} converges to \eqref{eq:EnergyFinite2ndMoment}. 

In the case when $\mu_{eq}$ has finite first moment but infinite second moment, we will need the following lemma controlling the optimal cost $C(\mu,\nu)$ under perturbations of $\mu$ and $\nu$. To state it, we first introduce the following terminology: For two measures $\nu$ and $\nu'$, let $R(\nu,\nu')\in [0,\infty]$ be the quantity
\begin{equation} \label{eq:RDef} R(\nu,\nu') = \inf_\gamma \esssup_{\gamma}|p-q| \end{equation}
where $\esssup_{\gamma}|p-q|$ denotes the essential supremum of the function $(p,q)\mapsto |p-q|$ with respect to $\gamma$ and the infimum is taken over all probability measures $\gamma$ on $\mathbb R^n\times \mathbb R^n$ with margins $\nu$ and $\nu'$. Note that for those $\gamma$ that are induced by a map $T:\mathbb R^n\rightarrow \mathbb R^n$, the right hand side of \eqref{eq:RDef} takes the form $\esssup_{\nu}|p-T(p)|.$

\begin{lemma}
\label{lem:CLessThanW1}
Let $\mu, \mu', \nu$ and $\nu'$ be probability measures on $\mathbb R^n$. Assume that $\mu$ and $\mu'$ have finite first moment and $\nu$ and $\nu'$ have compact support. Then
$$ |C(\mu,\nu) -C(\mu',\nu)| \leq  W_1(\mu,\mu')\cdot\esssup_{\nu}|p| $$
and
$$ |C(\mu,\nu) -C(\mu,\nu')| \leq \int_{\mathbb R^n} |x| \mu \cdot R(\nu,\nu') $$
where $\esssup_\nu |p|$ denotes the essential supremum of the function $p\mapsto |p|$ with respect to $\nu$. 
\end{lemma}
\begin{proof}
To prove the first inequality, let $\gamma_{\mu,\mu'}$ be a probability measure on $\mathbb R^n \times \mathbb R^n$ with marginals $\mu$ and $\mu'$ such that
$$ W_1(\mu,\mu') = \int_{\mathbb R^n\times \mathbb R^n} |x-y| \gamma_{\mu,\mu'}(x,y) $$
and $\gamma_{\mu',\nu}$ be a probability measure on $\mathbb R^n \times \mathbb R^n$ with marginals $\mu'$ and $\nu$ such that
$$ C(\mu',\nu) = -\int_{\mathbb R^n\times\mathbb R^n} \langle y,p\rangle \gamma_{\mu',\nu}(y,p). $$
For $1\leq i<j\leq 3$, let $\pi_{ij}$ be the projection from $\mathbb R^n\times \mathbb R^n \times \mathbb R^n$ to the $i$'th and $j$'th factor. By the Gluing Lemma (see for example Chapter~1 in \cite{V}), there is a probability measure $\Gamma$ on $\mathbb R^n\times \mathbb R^n \times \mathbb R^n$ such that $(\pi_{12})_\# \Gamma=\gamma_{\mu,\mu'}$ and $(\pi_{23})_\#\Gamma=\gamma_{\mu',\nu}$. Let $\gamma_{\mu,\nu} = (\pi_{13})_\# \Gamma$ and note that the first and second marginals of $\gamma_{\mu,\nu}$ are given by $\mu$ and $\nu$, respectively, hence $\gamma_{\mu,\nu}$ is a transport plan from $\mu$ to $\nu$. We also note that 
$$\supp \Gamma \subset \mathbb R^n \times \mathbb R^n \times \supp \nu$$ 
This gives
\begin{eqnarray} 
C(\mu,\nu) - C(\mu',\nu) & \leq & -\int_{(\mathbb R^n)^2} \langle x,p\rangle \gamma_{\mu,\nu}(x,p) + \int_{(\mathbb R^n)^2} \langle y,p \rangle\gamma_{\mu',\nu}(y,p) \nonumber \\
& = & \int_{(\mathbb R^n)^3} \langle y-x,p\rangle \Gamma(x,y,p) \nonumber \\
& \leq & \int_{(\mathbb R^n)^3} |x-y|\cdot|p| \Gamma(x,y,p) \nonumber \\
& \leq & \int_{(\mathbb R^n)^3} |x-y|\Gamma(x,y,p)\cdot\esssup_{\nu}|p| \nonumber \\
& = & \int_{(\mathbb R^n)^2} |x-y|\gamma_{\mu,\mu'}(x,y)\cdot\esssup_{\nu}|p| \nonumber \\
& = & W_1(\mu,\mu')\cdot\esssup_{\nu}|p|. \nonumber 
\end{eqnarray}
Interchanging the roles of $\mu$ and $\mu'$ in the argument gives similarly
$$ C(\mu',\nu) - C(\mu,\nu) \leq W_1(\mu,\mu')\cdot\esssup_{\nu}|p|. $$
This proves the first inequality.

For the second inequality, let $\gamma_{\nu,\nu'}$ be a probability measure on $\mathbb R^n \times \mathbb R^n$ with marginals $\nu$ and $\nu'$ and $\gamma_{\mu,\nu'}$ be a probability measure on $\mathbb R^n \times \mathbb R^n$ with marginals $\mu$ and $\nu'$ such that
$$ C(\mu,\nu') = -\int_{\mathbb R^n\times \mathbb R^n} \langle x,q\rangle \gamma_{\mu,\nu'}(x,q). $$
As above, letting $\Gamma$ on $\mathbb R^n\times \mathbb R^n \times \mathbb R^n$ such that $(\pi_{13})_\# \Gamma=\gamma_{\mu,\nu'}$ and $(\pi_{12})_\#\Gamma=\gamma_{\nu,\nu'}$ and writing $\gamma_{\mu,\nu}=(\pi_{12})_\#\Gamma$ we get
\begin{eqnarray} 
C(\mu,\nu) - C(\mu,\nu') & \leq & -\int_{(\mathbb R^n)^2} \langle x,p\rangle \gamma_{\mu,\nu}(x,p) + \int_{(\mathbb R^n)^2} \langle x,q \rangle\gamma_{\mu',\nu}(x,q) \nonumber \\
& = & \int_{(\mathbb R^n)^3} \langle x,q-p\rangle \Gamma(x,p,q) \nonumber \\
& \leq & \int_{(\mathbb R^n)^3} |x|\cdot |p-q|\Gamma(x,p,q) \nonumber \\
%& \leq & \int |x| \Gamma(x,y,p)\cdot\sup_{(p,q)\in \supp \gamma_{\nu,\nu'}}|p-q| \nonumber \\
& \leq & \int_{(\mathbb R^n)^3} |x| \Gamma(x,y,p)\cdot\esssup_{\gamma_{\nu,\nu'}}|p-q| \nonumber \\
& = & \int_{\mathbb R^n} |x|\mu\cdot\esssup_{\gamma_{\nu,\nu'}}|p-q|. \nonumber 
\end{eqnarray}
Since this holds for any $\gamma_{\nu,\nu'}$ with marginals $\nu$ and $\nu'$, this proves the lemma. 
\end{proof}

The significance of $R(\cdot,\cdot)$ for us is the following lemma: 
\begin{lemma}\label{lem:TransportMapConvergence}
Let $p^j\in (\mathbb R^n)^{N_j}$ and $\nu_j$ be defined as above. Then
$$R\left(\delta^{N_j}\left(\frac{p^j}{k}\right),\nu_j\right)\rightarrow 0$$
as $k\rightarrow \infty$.
\end{lemma}
\begin{proof}
For each $k\geq 1$, let $f_{j,N}:\mathbb R^n\rightarrow \mathbb R$ be the unique Alexandrov solution to the real Monge-Ampère equation 
$$ \det\left(D^2 f\right) = \delta^{N_j}\left(\frac{p^j}{k}\right) $$
such that $\nabla f_{j,N}(\mathbb R^n) = P_j$ and $f_{j,N}(0) = 0$.  Similarly, let $f_j:\mathbb R^n\rightarrow \mathbb R$ be the unique solution to 
\begin{equation} \label{eq:IdentityEquation} \det\left(D^2 f\right) = \nu_j \end{equation}
such that $\nabla f_j = P_j$ and $f_j(0)=0$. Note that $f_j(p) = |p|^2/2$ on the interior of $P_j$. By the Arcela-Ascoli theorem, we may extract a subsequence of $\{f_{j,N}\}_{k=1}^\infty$ which converges locally uniformly. By continuity of the real Monge-Ampère operator, any limit point of this sequence has to satisfy \eqref{eq:IdentityEquation}. Consequently, $f_{j,N}\rightarrow f_N$ locally uniformly. In particular, this convergence is uniform on $P_j$. By standard properties of convex functions, this means $D f_{j,N} \rightarrow D f_j = id$ uniformly on the interior of $P_j$. Finally, since $P_j$ is convex, $\nu_j$ does not assign any mass to the boundary of $P_j$. Consequently, 
$$ \esssup_\nu |p-D f_{j,N}(p)| \rightarrow 0 $$
as $k\rightarrow \infty$. As $(Df_{j,N}^*)_\#\nu_j = \delta^{N_j}\left(\frac{p^j}{k}\right)$ we get that the induced transport plan $\gamma = (Df_{j,N}^*\times id)_\# \nu_j$ is a candidate in the infimum defining $R(\nu_j,\delta^{N_j}(p^j))$. This proves the lemma. 
\end{proof}

% \begin{lemma}
% \label{lem:CostConv}
% Assume $\{\mu_i\}$ is a sequence of measures on $\mathbb R^n$ with finite first moments converging to $\mu$ in $W_1$. Then 
% $$\lim_{k\rightarrow \infty} E_\phi^*(\frac{1}{N_j}\sum_{j=1}^m\delta_{x_i}) = E_{\phi}(\mu). $$
% \end{lemma}
Using Lemma~\ref{lem:LowerBoundCost} we can now prove Proposition~\ref{prop:ToricLConv}.
\begin{proof}[Proof of Proposition~\ref{prop:ToricLConv}]
By Lemma~\ref{lem:Finite1stMoment}, the measure $\pi_\#(\omega^n)$ has finite first moment. 
As in the discussion after Lemma~\ref{lem:LowerBoundCost}, let $\{x_i\}_{i=1}^\infty$ be a sequence in $\mathbb R^n$ such that $\delta^{N}(x_1,\ldots,x_N) \rightarrow \mu_{eq}$ in the Wasserstein 1-metric. By Lemma~\ref{lem:LowerBoundCost}, 
\begin{eqnarray}
\limsup_{k\rightarrow \infty} \mathcal L_k(\phi_1,\ldots,\phi_m) & \leq & \limsup \sum_{j=1}^m -C\left(\delta^{N_j}(x_1,\ldots,x_{N_j}), \delta^{N_j}\left(\frac{p^j}{k}\right)\right) \nonumber \\
& & -\int_{\mathbb R^n} \sum_{j=1}^m \delta^{N}(x_1,\ldots,x_N)+o(k). \nonumber
\end{eqnarray}
The integral in the right hand side of this converges to $\int_{\mathbb R^n} \sum_{j=1}^m \phi_j \mu_{eq}$ since $\sum_{j=1}^m \phi_j(x)<C|x|$ for some $C$ that only depends on $L_1.\ldots,L_m$. Moreover, we claim that 
$$ C\left(\delta^{N_j}(x_1,\ldots,x_{N_j}), \delta^{N_j}\left(\frac{p^j}{k}\right)\right) \rightarrow C(\mu_{eq},\nu_j)$$ for each $j\in \{1,\ldots,m\}$. To see this, write 
\begin{eqnarray} 
& \left|C\left(\delta^{N_j}(x_1,\ldots,x_{N_j}), \delta^{N_j}\left(\frac{p^j}{k}\right)\right) - C(\mu_{eq},\nu_j)\right| & \nonumber \\
& \leq & \nonumber \\
& \left|C\left(\delta^{N_j}(x_1,\ldots,x_{N_j}), \delta^{N_j}\left(\frac{p^j}{k}\right)\right) - C\left(\mu_{eq}, \delta^{N_j}\left(\frac{p^j}{k}\right)\right)\right|  & \nonumber \\
& + & \nonumber \\
& \left|C\left(\mu_{eq}, \delta^{N_j}\left(\frac{p^j}{k}\right)\right)- C(\mu_{eq},\nu_j)\right| & \nonumber 
\end{eqnarray}
and note that by Lemma~\ref{lem:CLessThanW1}, the first term of this can be bounded by 
\begin{equation}
    \label{eq:muvariation}
    \esssup_{\delta^N(p^j)} |p| \cdot W_1(\delta^{N_j}(x_1,\ldots,x_{N_j}),\mu_{eq})
\end{equation}
and the second term can be bounded by
\begin{equation}
    \label{eq:nuvariation}
    \int_{\mathbb R^n} |x| \mu_{eq} \cdot R\left(\delta^{N_j}\left(\frac{p^j}{k}\right),\nu_j\right).
\end{equation}
The quantity in \eqref{eq:muvariation} vanish since $\supp \delta^{N_j}\left(\frac{p^j}{k}\right) \subset P_j$ for each $k$ and the quantity in \eqref{eq:nuvariation} vanishes by Lemma~\ref{lem:TransportMapConvergence}. 

Now, using \eqref{eq:ToricEnergyCostForm}, we get
$$ \limsup_{k\rightarrow \infty} \mathcal L_k(\phi_1,\ldots,\phi_m) \leq F\left(\sum_{j=1}^m \phi_j\right).$$
Together with Lemma~\ref{lem:LUpperBound}, this implies the first bullet in Theorem~\ref{thm:EqCond}.
\end{proof}
\begin{remark}
An essential point in the proof of Proposition~\ref{prop:ToricLConv} is the fact that we can find one sequence $\{x_i\}_{i=1}^\infty$ which works for all $j\in \{1,\ldots,m\}$. In this sense, the sequence $\{x_i\}_{i=1}^\infty$ can be thought of as a real analog of a mutual asymptotic Fekete sequence for $(L_1,\psi_1),\ldots,(L_m,\phi_m)$. 
\end{remark}
\begin{proof}[Proof of Theorem~\ref{thm:ToricLConv}]
By Theorem~2 in \cite{H21a}, the equilibrium measure $$\mu_{(L_1,\ldots,L_m,\sum_{j=1}^m \phi_j)}$$ is absolutely continuous with bounded density. This means the integrability assumption in Proposition~\ref{prop:ToricLConv} is satisfied, hence the first bullet in Theorem~\ref{thm:EqCond} holds as long as 
\begin{equation} 
\label{eq:VolumeAssumption}
L_1^n=\ldots=L_m^n = 1.
\end{equation}
Noting that the second bullet in Theorem~\ref{thm:EqCond} is invariant under scaling of $(L_j,\phi_j)$ for each $j\in \{1,\ldots,m\}$, we see that the condition in \eqref{eq:VolumeAssumption} can be dropped.
\end{proof}

\subsection{Expanding the determinants}
We now turn to the proof of Lemma~\ref{lem:LowerBoundCost}, beginning with a birds eye view of the argument. 
%\begin{lemma}
%\label{lem:TropPerm}
%Let $z=x+iy$ and $\phi = \sum_{j=1}^m\phi_j$. For each $j$, let $\sigma_j$ be a permutation of $\{1,\ldots,N_j\}$. Then 
%\begin{equation} 
%\label{eq:TropPerm}
%\int_{X^N} \left|\prod_{j=1}^k D_j(z_1,\ldots,z_N)\right|^2_\phi \geq \int_{(\mathbb R^n)^N} \prod_{j=1}^k e^{\sum_i \langle x_i,p^j_{\sigma_j(i)}\rangle-\phi_j(x_i)}. 
%\end{equation}
%where, for each $j$, $\sigma_j\in S_{N_j}$ and satisfies
%$$ \langle x_i,p^j_{\sigma_j(i)}\rangle-\phi_j(x_i) = \max_{\sigma\in S_{N_j}} \langle x_i,p^j_{\sigma(i)}\rangle-\phi_j(x_i). $$
%\end{lemma}
Expanding the determinants in \eqref{eq:ToricLk} we get a sum consisting of one term for each pair of vectors 
$$ (\sigma_1,\ldots,\sigma_m), (\sigma_1',\ldots,\sigma_m')$$ 
where each $\sigma_j$ and $\sigma_j'$ are elements in the group $S_{N_j}$ of permutations of $\{1,\ldots,N_j\}$. In the case $m=1$, where we get one term for each pair of permutations $\sigma_1,\sigma_1'$, integrating over the fibers of the map $\pi:(\mathbb C^*)^n\rightarrow \mathbb R^n$ eliminates all terms corresponding to pairs where $\sigma\not=\sigma'$. This is exploited in \cite{BPP} and \cite{HPP} and the result is a sum of \emph{positive} terms consisting of one term for each permutation $\sigma_1$. We will apply the same trick, but when $m>1$ we get a number of 'cross terms', some of them positive and some of them negative. The following lemma expands the determinants, integrates over the fibers of $\pi$ and rewrites the resulting sum into a sum of \emph{non-negative} terms. Before we state the lemma we note that if 
$$ x+\sqrt{-1}y = (x_1,\ldots,x_N)+\sqrt{-1}(y_1,\ldots,y_N)$$ 
are logarithmic coordinates on $((\mathbb C^*)^n)^N$ and $dx$ and $dy$ are the Lebesgue measures on $(\mathbb R^n)^N$ and $([0,2\pi]^n)^N$, respectively, then 
$$ (\omega^n)^{\otimes N} = fdx\otimes dy $$ 
where $f$ is a smooth function $((\mathbb C^*)^n)^N$ that only depends on $x$. 
\begin{lemma}\label{lem:LkCollectedTerms}
Given a vector $a=(a_1,\ldots,a_N) \in (\bbZ^n)^N$, let $S_a$ be the set of $(\sigma_1,\ldots,\sigma_m)\in \prod_{j=1}^m S_{N_j}$ such that
\begin{equation} \label{eq:adef} \sum_{j=1}^m p_{\sigma_j(i)} = a_i \end{equation}
for all $i\in \{1,\ldots,N\}$. Then
\begin{eqnarray} \int_{([0, 2\pi]^n)^N} \prod_{j=1}^m |D_j(x+\sqrt{-1}y)|^2_{\phi_j}dy & = & \sum_{a\in (\bbZ^n)^N} C_a e^{2\sum_{i=1}^N\left(\left\langle x_i,  a_i \right\rangle-k\phi(x_i)\right)} \nonumber \\ 
& & + o_k(N) 
\label{eq:LkCollectedTerms}
\end{eqnarray}
% \begin{eqnarray} \mathcal L_k(\phi_1,\ldots,\phi_m) & = & \int_{(\mathbb R^n)^N}\left( \sum_{a\in (\bbZ^n)^N} C_a \prod_{i=1}^N e^{2\sum_{j=1}^m\left\langle x_i,  a_i \right\rangle-k\phi(x_i)}\right)(\pi_\# \omega^n)^N \nonumber \\ 
% & & + o(k) 
% \label{eq:LkCollectedTerms}
% \end{eqnarray}
where 
$$ C_a := \left(\sum_{(\sigma_1,\ldots,\sigma_m)\in S_a} (-1)^{\sigma_1 \cdots \sigma_m}\right)^2 \geq 0. $$
%and $\pi_\# \omega^n$ is the push forward of the measure $\omega^n$ under $\pi$. 
\end{lemma}
\begin{proof}
% Let $\pi^{\times N}:((\bbC^*)^n)^N \rightarrow \bbR^n$ be the map given by 
% $$ \pi^{\times N}(w_1,\ldots,w_N)=(\pi(w_1),\ldots,\pi(w_N)) = (x_1,\ldots,x_N). $$
% The lemma will be proved by establishing the following inequality for each $x\in (\bbR^n)^N$:
% \begin{eqnarray} & & \int_{\left(\pi^{\times N}\right)^{-1}(x)} \prod_{j=1}^m \left|\det\left[e^{\langle x_i+\sqrt{-1}y_i,p_l \rangle} \right]_{i,l}\right|^2e^{-k\sum_{i} \phi(x_i)} \left(\left.\mu_0^{\otimes N}\right|_{\left(\pi^{\times N}\right)^{-1}(x)}\right) \nonumber \\ 
% & \geq & \prod_{j=1}^k e^{\sum_i \langle x_i,p^j_{\sigma_j(i)}\rangle-\phi_j(x_i)}.
% \label{eq:FiberIneq} 
% \end{eqnarray}
Expanding the determinants in \eqref{eq:ToricLk}, we get
\begin{eqnarray}
 & & \prod_{j=1}^m \left|\det\left(e^{\langle x_i+\sqrt{-1}y_i,p_l^j \rangle} \right)_{i,l}\right|^2 \nonumber \\
 & = &  \prod_{j=1}^m \left(\sum_{\sigma_j\in S_{N_j}}(-1)^{\sigma_j}\prod_{i=1}^{N_j} e^{\langle x_i+\sqrt{-1}y_i,p_{\sigma_j(i)}^j\rangle}\right)\overline{\left(\sum_{\sigma_j'\in S_{N_j}}(-1)^{\sigma_j'}\prod_{i=1}^{N_j} e^{\langle x_i+\sqrt{-1}y_i,p_{\sigma_j'(i)}^j\rangle}\right)}  \nonumber \\
& = & \sum_{\substack{\sigma_1,\sigma_1'\in S_{N_1} \\ \vdots \\ \sigma_m,\sigma_m'\in S_{N_m}}} (-1)^{\sigma_1 \sigma_1'\cdots \sigma_m \sigma_m'} \prod_{i=1}^N e^{\left\langle x_i, \sum_{j=1}^m \left(p_{\sigma_j(i)}^j+p_{\sigma_j'(i)}^j\right) \right\rangle + \sqrt{-1}\left\langle y_i, \sum_{j=1}^m \left(p_{\sigma_j(i)}^j-p_{\sigma_j'(i)}^j\right) \right\rangle}  \nonumber \\
& & \label{eq:DetExpanded}
\end{eqnarray}
where we have tacitly assumed $ p^j_{\sigma_j(i)} = p^j_{\sigma_j'(i)} = 0 $ if $i>N_j$. Integrating the $i$'th factor of the product in \eqref{eq:DetExpanded} in the variable $y_i$ on $[0,2\pi]$ we get 
$$ \begin{cases} 
2\pi e^{2\sum_{j=1}^m \left\langle x_i, p_{\sigma_j(i)}\right\rangle} & \text{ if } \sum_{j=1}^m \left(p_{\sigma_j(i)}-p_{\sigma_j'(i)}\right) = 0 \\
0 & \text{ otherwise.}
\end{cases} $$
Performing this integration for each $i\in \{1,\ldots,N\}$ we get that the left hand side of \eqref{eq:LkCollectedTerms} can be written as 
\begin{equation} \label{eq:RealIntegral} (2\pi)^N \sum_{\substack{\sigma_1,\sigma_1' \\ \vdots \\ \sigma_m,\sigma_m'}} (-1)^{\sigma_1 \sigma_1'\cdots \sigma_m \sigma_m'} \prod_{i=1}^N e^{2\sum_{j=1}^m\left\langle x_i,  p_{\sigma_j(i)} \right\rangle}  \end{equation}
where the sum is taken over all $\sigma_1,\sigma_1',\ldots,\sigma_m,\sigma_m'$ such that $\sigma_j,\sigma_j'\in S_{N_j}$ for each $j\in \{1,\ldots,m\}$ and 
\begin{equation} \label{eq:SigmaSigmaPrimeCond} \sum_{j=1}^m \left(p_{\sigma_j(i)}-p_{\sigma_j'(i)}\right) = 0 \end{equation}
for each $i\in \{1,\ldots,N\}$.

Note that \eqref{eq:SigmaSigmaPrimeCond} holds for all $i\in \{1,\ldots,N_j\}$ if and only if $(\sigma_1,\ldots,\sigma_m)\in S_a$ and $(\sigma_1',\ldots,\sigma_m')\in S_a$ for some $a\in (\mathbb Z^n)^N$. Noting also that 
$$ \sum_{\substack{(\sigma_1,\ldots,\sigma_m)\in S_a\\ (\sigma_1',\ldots,\sigma_m') \in S_a}} (-1)^{\sigma_1 \sigma_1'\cdots \sigma_m \sigma_m'} = \left(\sum_{(\sigma_1,\ldots,\sigma_m)\in S_a} (-1)^{\sigma_1 \cdots \sigma_m}\right)^2 $$
we get that \eqref{eq:RealIntegral} can be replaced by
$$ \sum_{a\in (\mathbb Z^n)^N} \left(\sum_{(\sigma_1,\ldots,\sigma_m)\in S_a} (-1)^{\sigma_1 \cdots \sigma_m}\right)^2 e^{2\sum_{i=1}^N \langle x_i,a_i \rangle}. $$
This proves the lemma. 
\end{proof}

Note that 
\begin{eqnarray}
\min_{\substack{a\in (\mathbb Z^n)^N \\ S_a \not= \emptyset}} -\sum_{i=1}^N\langle x_i,a_i\rangle
& = & \min_{\substack{(\sigma_1,\ldots,\sigma_m) \\ \in \\ \prod_{j=1}^m S_{N_j}}} -\sum_{j=1}^m\sum_{i=1}^{N_j}\langle x_i,p_{\sigma_j(i)}^j\rangle \nonumber \\
 & = & \sum_{j=1}^m\min_{\sigma\in S_{N_j}}-\sum_{i=1}^{N_j}\langle x_i,p_{\sigma(i)}^j\rangle  \nonumber \\
 & = & \sum_{j=1}^m kN_j\min_{\sigma\in S_{N_j}} -\frac{1}{N_j}\sum_{i=1}^{N_j}\left\langle x_i, \frac{p^j_{\sigma_j(i)}}{k}\right\rangle \nonumber \\ 
 & = & \sum_{j=1}^m kN_jC\left(\delta^{N_j}(x_1,\ldots,x_{N_j}),\delta^{N_j}\left(\frac{p^j}{k}\right)\right). \nonumber \\
 &  & \label{eq:mina}
\end{eqnarray}
Consequently, in order to prove Lemma~\ref{lem:LowerBoundCost}, we would like to bound the right hand side of \eqref{eq:LkCollectedTerms} from below by 
\begin{equation} \max_{a\in (\mathbb Z^n)^N, \, S_a \not= \emptyset} e^{2\sum_{i=1}^N\left(\langle x_i,a_i\rangle-k\phi(x_i)\right)}+o(N). \label{eq:MaxTerm}\end{equation}
Let $b\in (\mathbb Z^n)^N$ be an element where the max in \eqref{eq:MaxTerm} is attained. Then this bound follows directly as long as $C_b>0$. In particular, this is true when $S_b$ is a singleton. It will follow from the following lemma that $S_b$ is a singleton for almost all $x\in (\mathbb R^n)^N$.
\begin{lemma}
\label{lem:UniqueOptMap}
Let $p=(p_1,\ldots,p_{N})\in (\mathbb R^n)^{N}$ The set of $x=(x_1,\ldots,x_{N})\in (\mathbb R^n)^N$ such that the optimal transport problem from $\delta^{N}(x)$ to $\delta^{N}(p)$ with respect to the cost function $c(x,y)=-\langle x,y\rangle$ admits more than one optimal map is contained in a finite union of hyperplanes of $(\mathbb R^n)^N$. 
\end{lemma}
\begin{proof}
Assume the optimal transport problem from $\sum_{j=1}^m\delta_{x_i}$ to $\sum_{j=1}^m\delta_{p_i}$ admits more than one optimal map, i.e. there are two permutations, $\sigma_1$ and $\sigma_2$, of $\{1,\ldots,N\}$ such that 
$$ -\sum_{i=1}^N \langle x_i,p_{\sigma_1(i)} \rangle = -\sum_{i=1}^N \langle x_i,p_{\sigma_2(i)} \rangle= \inf_\sigma \left(-\sum_{j=1}^m\langle x_i,p_{\sigma(i)}\rangle\right). $$
This means $\sum_{i=1}^N \langle x_i,p_{\sigma_1(i)}-p_{\sigma_2(i)}\rangle=0$, hence that $(x_1,\ldots,x_N)$ lies in the finite union of hyperplanes given by
$$ \bigcup_{\sigma,\sigma'} \left\{(x_1,\ldots,x_N)\in X^N: \sum_{i=1}^N \langle x_i,p_{\sigma(i)}-p_{\sigma'(i)} \rangle = 0 \right\} $$
where the union is taken over all permutations $\sigma$ and $\sigma'$ of $\{1,\ldots,N\}$.
\end{proof}

\begin{proof}[Proof of Lemma~\ref{lem:LowerBoundCost}]
First of all, we claim that the right hand side of \eqref{eq:LkCollectedTerms} is bounded from below by \eqref{eq:MaxTerm}. By Lemma~\ref{lem:LkCollectedTerms}, it suffice to prove that $C_b>0$, where $b\in (\mathbb Z^n)^N$ is an element where the max in \eqref{eq:MaxTerm} is attained. By \eqref{eq:mina}, for any $(\sigma_1,\ldots,\sigma_m)\in S_b$ we have that $\sigma_j$ defines an optimal map from $\delta^{N_j}(x_1,\ldots,x_{N_j})$ to $\delta^{N_j}(p^j/k)$ for each $j\in\{1,\ldots,m\}$. By Lemma~\ref{lem:UniqueOptMap}, this completely determines $\sigma_j$ for all $x\in (\mathbb R^n)^N$ except for a finite union of hyperplanes. By continuity of the left hand side of \eqref{eq:LkCollectedTerms} with respect to $x$ (which can be established using the Dominated Convergence Theorem) and continuity of \eqref{eq:MaxTerm} with respect to $x$, the claim holds for all $x$.

Using that the right hand side of \eqref{eq:LkCollectedTerms} is bounded from below by \eqref{eq:MaxTerm}, it suffices to show that 
$$ \liminf_{k\rightarrow \infty} \frac{1}{kN} \log \int_{(\mathbb R^n)^N} \max_{a\in (\mathbb Z^n)^N, \, S_a \not= \emptyset} e^{2\sum_{i=1}^N\left(\langle x_i,a_i\rangle-k\phi(x_i)\right)} dx $$
is bounded from below by
$$  \liminf_{k\rightarrow \infty} \frac{1}{kN}\log \max_{a\in (\mathbb Z^n)^N, \, S_a \not= \emptyset} e^{2\sum_{i=1}^N\left(\langle x_i,a_i\rangle-k\phi(x_i)\right)} $$
for any $x\in (\mathbb R^n)^N$. This follows from the fact that 
$$ \max_{a\in (\mathbb Z^n)^N, \, S_a \not= \emptyset} e^{2\sum_{i=1}^N\left(\langle x_i,a_i\rangle-k\phi(x_i)\right)}$$ 
is subharmonic and a similar application of the meanvalue property as in the proof of Lemma~\ref{lem:L2LInf}. This proves the lemma. 
\end{proof}

\begin{proof}[Proof of Theorem~\ref{thm:ToricAsyFekete}]
Equation \eqref{eq:SameEqMeasures} is a necessary condition by Lemma~\ref{lem:NesCond}. By the second bullet in Theorem~\ref{thm:EqCond} and Theorem~\ref{thm:ToricLConv}, it is also sufficient. 
%Note that the regularity condition in Theorem~\ref{thm:ToricLConv} is clearly satisfied if $\mu_{(L_1,\ldots,L_m,\sum_{j=1}^m\phi_j)}$ is absolutely continuous with bounded density. This follows from the regularity assumption in Theorem~\ref{thm:ToricAsyFekete} and Theorem~2 in \cite{BB}.
\end{proof}

\end{document}